\documentclass{article}
\usepackage[utf8]{inputenc}
\usepackage{todonotes}
\usepackage{caption}
\usepackage{nicefrac}
\usepackage{amsmath,enumerate}
 \usepackage{amsthm}
\usepackage{amsfonts}
\usepackage{amssymb}
\usepackage{mathrsfs}
\usepackage{pgfplots}
\usepackage{mathtools}
\usepackage{comment}
\usepackage{lmodern}
\usepackage{subfig}
\usepackage{caption}

\usepackage{tikz}
\usetikzlibrary{shapes,arrows,graphs,graphs.standard,shapes.misc}
\usetikzlibrary{matrix,decorations.pathreplacing, calc, positioning,fit}
\usetikzlibrary{shapes.geometric}

\newtheorem{Theorem}{Theorem}[section]
\newtheorem{Definition}{Definition}[section]
\newtheorem{Example}{Example}[section]
\newtheorem{Proposition}[Theorem]{Proposition}
\newtheorem{Lemma}{Lemma}[section]
\newtheorem{Corollary}[Theorem]{Corollary}
\newtheorem{Assumption}{Assumption}[section]
\newtheorem{Remark}{Remark}[section]

\newcommand{\R}{\mathbb{R}}

\DeclareMathOperator{\diag}{diag}

\newcommand{\bfo}{{\bf 1}}
\renewcommand{\sp}{\operatorname{splx}}
\newcommand{\cH}{{H}}

\title{Triangulated Laman Graphs, Local Stochastic Matrices, and Limits of Their Products}
\author{Mohamed Ali Belabbas \footnote{Coordinated Science Laboratory, University of Illinois, Urbana-Champaign. Email: \texttt{belabbas@illinois.edu}} and Xudong Chen\footnote{Department of Electrical, Computer and Energy Engineering, University of Colorado, Boulder. Email: \texttt{xudong.chen@colorado.edu}}}
\date{}

\begin{document}

\maketitle

\begin{abstract}
We derive conditions on the products of stochastic matrices guaranteeing the existence of a unique limit invariant distribution. Belying our approach is the hereby defined notion of restricted triangulated Laman graphs. The main idea is the following: to each triangle in the graph, we assign a stochastic matrix. Two matrices can be adjacent in a product only if their corresponding triangles share an edge in the graph. We provide an explicit formula for the limit invariant distribution of the product in terms of the individual stochastic matrices.  
\end{abstract}

\section{Introduction}

The issue of convergence of infinite product of (row) stochastic matrices arises naturally in the study of finite-state Markov chains and in the design of consensus algorithms. As a result, it has been widely investigated in the past decades~\cite{wolfowitz1963products, chatterjee1977towards,daubechies1992sets, lorenz2005stabilization, bolouki2011consensus,touri2012product,chevalier2018convergent} from a variety of perspectives.   
The main problem investigated in the above works is whether the limit of a left product $\lim_{k\to\infty} A_k\cdots A_2A_1$ converges to a rank one matrix $\mathbf{1} w^\top$, where $\mathbf{1}$ is a vector of all ones and $w$ is a probability vector, i.e., entries of $w$ are nonnegative and sum to $1$.

A less studied, yet critical, problem is to characterize the limit  beyond the fact that it is rank one. This amounts to the characterization of the probability vector $w$. 
In the context of Markov chain, $w$ is the limiting distribution while in the context of (weighted) consensus, entries of $w$ are the averaging weights in the convex combination.  The problem is hard to tackle. 
Indeed, barring simple cases such as using only commuting matrices, the limit depends on the order in which the  stochastic matrices appear in the infinite products. This is true even if the matrices appearing in the product are chosen from a finite set. 
See~\cite{daubechies1992sets} for some illustrations of the above mentioned dependence. 
Thus, without knowing the entire sequence a priori, it is in general infeasible to characterize the limit (provided that it exists). In fact, even if one knows the order of the entire sequence, the analysis for obtaining an explicit formula of the limit is often intractable.

In this paper, we address this latter problem, i.e., we characterize limits of certain products of stochastic matrices. We elaborate below on the type of products considered in the paper.  
As is usually done, we use a graph $G = (V, E)$ to denote the states of Markov chain and the allowable transitions between these states. 
In the context of consensus, the graph  represents the information-flow topology between different agents.   
We introduce a class of graphs, termed triangulated Laman graphs (TLGs), and use their structure to define sets of stochastic matrices and the orders in which we can take their products. Specifically, given any TLG, we assign a stochastic matrix to each triangle in the graph.  
The matrix can be obtained by starting with the identity matrix and, then, replacing the principal submatrix corresponding to the nodes in the triangle with an arbitrary $3 \times 3$ rank-one stochastic matrix. 
We call these matrices ``local stochastic matrices'' as the transition probabilities (or the communications in the context of consensus) involve only the nodes in that triangle.  To describe the allowable products of those local stochastic matrices, we introduce the notion of derived graph associated with a TLG. It is a graph whose nodes are the triangles of a TLG, and whose edges capture a notion of adjacency  between these triangles --- two triangles are adjacent if they share a common edge.  
The allowable products are then the ones for which adjacent matrices correspond to adjacent nodes in the derived graph.

A major contribution of the paper is to show  that if a walk in the derived graph  visits every node infinitely often, then the limit of the associated product is a rank-one matrix. Moreover, the limit depends only on the {\em first node} of the walk. Because the derived graph is finite, there can only be finitely many different limits.   
The result is formulated in Theorem~\ref{th:main1}, and a complete characterization of these limits is provided in Sec.~\ref{ssec:characprodlim}. 

There are several implications of the above result. For example, any simple random walk on these derived graphs yield a convergent product of local stochastic matrices with probability one and the limits are independent of the sample paths but for their starting nodes (Corollary~\ref{cor:randomwalk}). Another consequence of the result concerns absolute probability vectors (APVs), which were introduced in~\cite{kolmogoroff1936theorie} to study the convergence of products of stochastic matrices (We recall its definition in Def.~\ref{def:apv}). 
Generically, the sequence of APVs depends on a particular convergent product, and, moreover, takes infinitely many different values, even when only a finite number of distinct matrices appear in the product. In contrast, we show in Corollary~\ref{cor:finiteanduniqueAPVs} that one can assign to a TLG, together with a set of local stochastic matrices, a {\it finite} set of vectors such that the sequence of APVs attached to {\it any}  allowable, convergent product of these local stochastic matrices takes values only from that finite set.

A large part of the novelty of this work lies in the introduction of TLGs and their derived graphs, as one may observe from the above description. To characterize their properties, we will obtain a recursive construction for them. This construction is akin to the celebrated Henneberg sequence that appears in rigidity theory~\cite{graver1993combinatorial}, and we thus call it Restricted Henneberg Construction (RHC). We prove  that any TLG can be obtained by an RHC and, reciprocally, any RHC yields a TLG. The proof may be of independent interest---indeed, TLGs have also appeared in earlier work on formation control~\cite{chen2017global}---but because it uses a set of ideas distinct from the ones used in the main part of the paper, we relegate it to the Appendix.

What is perhaps the closest line of work, in spirit, to the present is the work on gossiping~\cite{boyd2006randomized,he2011periodic,liu2011deterministic}. A gossip can be described, in terms of message passing, as an operation in which two agents communicate their values to each other and take the average. When described in terms of stochastic matrices, this yields a matrix which is the identity save for a 2-by-2 principal submatrix whose entries are $1/2$. It is shown that the left-product of such stochastic matrices converges, under some conditions, to the matrix with all entries $\frac{1}{n}$.  
More recently, it has been extended to clique gossiping~\cite{liu2019clique}, where $k$ agents in a clique perform an averaging operation. In these works, the convergence to the averaging matrix is a by-product of the fact that the matrices involved are in fact {\it doubly-stochastic}, i.e., all the row sums and column sums of the matrix are one.  

In terms of applications to consensus, besides the fact that our work allows for a control of the limiting probability vector while requiring minimal information about the allowable sequence (namely, the starting node), it also enables the implementation of simple {\it secure-by-design} consensus algorithms. 
Indeed, small networks are by nature more secure than larger networks, since by definition they contain fewer possible points of failure or attack. The smallest meaningful network in our case is the triangle. The local stochastic matrices are so that after each iteration, each node of the triangle have to agree on the same value. Furthermore, the adjacency rule is so that the next triangle to update has {\it two} nodes in common with the previous triangle. Hence the third node in the triangle can verify that it receives the same value from the other two nodes. This built-in redundancy adds an obvious layer of security to the updates and complements some existing secure consensus algorithm, e.g.~\cite{wang2019resilient}, but of course does not make them impervious to tampering.

The remainder of the paper is organized as follows: we end this section by introducing key notations and terminologies used throughout the paper. In Sec.~\ref{sec:tlg}, we introduce the basic objects used in the paper: namely, triangulated Laman graphs, their derived graphs, and local stochastic matrices. Several key properties will be established in the section as well. Next, in Sec.~\ref{sec:mainresults}, we state the main results of the paper, including an explicit formula for the limits of allowable convergent products. In Sec.~\ref{sec:proofs}, we prove the main results, save for Theorem~\ref{th:equihennebergrtl} concerning the construction of TLGs, which we relegate to the appendix. Numerical studies are provided in Sec.~\ref{sec:numerical}, validating the main results and showing that they do not hold if some of the assumptions are broken.  
The paper ends with conclusions.

\paragraph{Notations and conventions}
We denote by $G=(V,E)$ be a graph, with node set $V$ and edge set $E$. All graphs considered in the paper are simple, i.e., there have no self-arc. 
We use $v_i$ to denote a node of $G$. If $G$ is {\it undirected}, we denote an edge by  $(v_i,v_j)$, and if $G$ is {\it directed}, we denote an edge {\em from} $v_i$ {\em to} $v_j$ by $v_iv_j$ . We refer to $|V|$ as the size of $G$. 
Given a subset of nodes $V' \subseteq V$,  the subgraph of $G$ induced by $V'$ is defined as $G'=(V',E')$ where $E' =\{(v_i,v_j)\mid v_i,v_j \in V' \mbox{  and } (v_i,v_j) \in E\}$ (resp. $E' =\{v_iv_j\mid v_i,v_j \in V' \mbox{  and } v_iv_j \in E\}$).

We call a sequence of nodes $\gamma=v_1\cdots v_k$  a {\it walk} in $G$ 
if $(v_i, v_{i+1})$ (resp. $v_iv_{i+1}$) is an edge of $G$, for all $1 \leq i \leq k-1$. We denote by $\gamma \vee v_*$ the walk $v_1\cdots v_kv_*$ where $(v_k, v_*)$ (resp. $v_kv_*$) needs to be an edge in $G$ for the operation to be well-defined. We denote by $\gamma^{-1}$ the reverse walk $v_k\cdots v_1$.

We say that $\gamma$ is a {\it closed walk} if $\gamma$ is a walk with $\gamma_1=\gamma_k$. We emphasize that for our purpose, a closed walk has a well-defined starting node. A path is a walk without repetition of nodes. A cycle is a closed path, i.e., only the starting node and ending node are repeated. 
The {\it length} of $\gamma$ is the number of edges traversed by $\gamma$, counted with multiplicity. 
The {\it cardinality} of $\gamma$, denoted by $|\gamma|$, is the number of nodes in $\gamma$, counted with multiplicity as well.

A triangle in a graph is a cycle of length $3$. We denote triangles using the letter $\Delta$, and describe them as the sets of their constituent nodes, e.g., $\Delta=\{v_i,v_j,v_k\}$ and we can write $v_i \in \Delta$.

We denote by $\{e_1,\ldots, e_n\}$ the standard basis in $\R^n$. Denote by $\bfo_n$ the vector of all $1$'s of dimension $n$. We omit the index when the dimension is clear from the context. 
For any vector $w = (w_1,\ldots, w_n)$, we use shorthand notation 
$\min{w}:=\min_{1\le i \le n} w_i$.
We call $w$ a positive (resp. nonnegative) vector if each entry $w_i$ is positive (resp. nonnegative), and $w$ a {\it probability vector} if it is nonnegative and its entries sum to $1$. 
We denote by $\sp(n-1)$ the standard simplex in $\R^n$, which is comprised of all probability vectors.

\section{Triangulated Laman Graphs}\label{sec:tlg}

We present in this section a class of graphs, termed {\it  Triangulated Laman Graphs} (TLGs), as well as a simple iterative algorithm to construct them, termed {\it restricted Henneberg construction} (RHC). 

In order to introduce the TLGs, we first recall that a graph $G=(V,E)$ is said to be {\it triangulated} if for every cycle of length strictly greater than $3$, there is an edge joining two nonconsecutive vertices of the cycle.
We call any cycle of length $3$ a {\it triangle}. 
Any edge $e$ that belongs to only {\it one} triangle is called {\it simple.} TL graphs are also minimally rigid,  see the Appendix or~\cite{graver1993combinatorial} for a formal definition. We thus include ``Laman'' explicitly in the definition:  

\begin{Definition}[Triangulated Laman Graphs (TLGs)]\label{def:rtl}
A graph $G$ is a {\bf  triangulated Laman Graph} (TLG) if it is both triangulated and minimally rigid. 
\end{Definition}

Because all minimally rigid graphs can be obtained by a so-called Henneberg construction~\cite{graver1993combinatorial}, so can be all TLGs. However, not every Henneberg construction gives rise to a TLG. 
We now introduce below restricted Henneberg constructions that produce TLGs and have the property that {\it all} TLGs can be obtained by such construction:          
\begin{description}
\item[Initialization:] Start from a graph $G_3=(V_3,E_3)$  with $V_3=\{v_1,v_2,v_3\}$,  $E_3=\{(v_1,v_2),(v_2,v_3),(v_1,v_3)\}$. It consists of one triangle. 
\item[Inductive step:] Suppose that a subgraph $G_k$ of $k$ nodes $v_1,\ldots, v_k$ has been constructed. Pick an edge $e=(v_l,v_m)$ in $G_k$. 
Add a node $v_{k+1}$ and two edges $(v_l,v_{k+1}), (v_m,v_{k+1})$ to $G_k$ to obtain $G_{k+1}$
\end{description}

\begin{Definition}
We refer to the above construction as a {\bf restricted Henneberg construction} {\em ({\bf RHC})}.
\end{Definition}

We have the following result:   

\begin{Theorem}\label{th:equihennebergrtl}
A graph is a TLG if and only if it can be constructed by a restricted Henneberg construction. 
\end{Theorem}

A proof of the theorem is provided in the Appendix. 

\paragraph{Derived graphs and their properties}
We now introduce the notion of derived graph associated with a triangulated graph $G$. Roughly speaking, the derived graph is used to reflect the adjacency of triangles in $G$, see Fig.~\ref{fig:exampledg} for an illustration.

\begin{Definition}[Derived graph]\label{def:derived}
Let $G$ be a triangulated graph. The derived graph $D_G$ of $G$ is an  undirected graph defined as follows: Each node $\Delta_i$ of $D_G$ corresponds to a triangle of $G$. If two distinct triangles corresponding to $\Delta_i$ and $\Delta_j$ share a common edge in $G$, then an edge $(\Delta_i, \Delta_j)$ is in $D_G$.  
\end{Definition}

Throughout the paper,  
we will view $\Delta_i$ both as a node of $D_G$, and as a subgraph of $G$---more precisely, a subgraph induced by three adjacent nodes. 
We will write $v \in \Delta_i$ (resp. $e \in \Delta_i$) to denote that the vertex $v$ (resp. edge $e$) is in the subgraph $\Delta_i$.

\begin{figure}
\centering
\subfloat[\label{fig:nonrtl}]{
\includegraphics{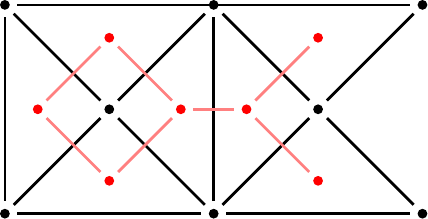}
}\qquad\qquad 
\subfloat[\label{fig:rtl}]{
\includegraphics{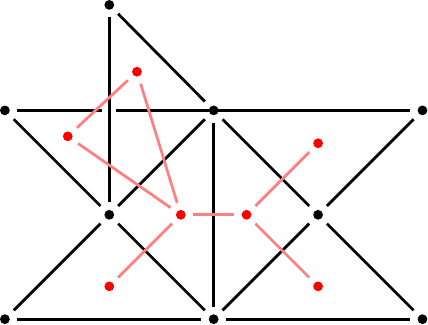}
}
\caption{Two triangulated graphs (in black) and their derived graphs (in red). The graph on the left is not a TLG, as it contains $8$ nodes and $14$ edges. The graph on the right is a TLG.  }\label{fig:exampledg}
\end{figure}

We next establish a few relevant properties for the derived graphs of TLGs. We start with the following fact: 

\begin{Proposition}\label{prop:tl}
Let $G$ be a TLG on $n$ nodes. Then, there are  $(n-2)$ triangles in $G$ and the derived graph $D_G$ is a connected, triangulated graph on $(n - 2)$ nodes.
\end{Proposition}

\begin{proof}
It should be clear from the RHC that $G$ has $(n-2)$ triangles and that $D_G$ is connected. We show below that $D_G$ is triangulated. 
The proof will be carried out by induction on the number of nodes in $G$. For the base case $n=3$, $G$ contains one triangle and $D_G$ is comprised of a single node. This proves the base case.

For the inductive step, we assume that the statement holds for $(n-1)$ and prove it for $n$. 
Let $G$ be a TLG on $n$ nodes. By Theorem~\ref{th:equihennebergrtl}, it admits an RHC. Following this RHC up to step $(n-3)$ yields a TLG $G'$ on $(n-1)$ nodes, which is a {\it subgraph} of $G$. By the induction hypothesis, the derived graph $D_{G'}$ is triangulated. We now focus on the last step of the RHC, yielding $G$ from $G'$. Denote by $e=(v_l,v_m)$ the edge in $G'$ selected, and by $v_n$ the newly added node. Denote by $\Delta_{i_1},\ldots,\Delta_{i_p}$ the triangles in $G'$ that contain the edge $e$. Then, the subgraph of $D_{G'}$ induced by these nodes is the complete graph $K_p$.

The newly added triangle $\Delta_{n-2}=\{v_l,v_m,v_n\}$ is a node in $D_G$.  It is connected in $D_G$ to {\it all} the nodes $\Delta_{i_1},\ldots,\Delta_{i_p}$. We thus conclude that the subgraph of $D_G$ induced by $\Delta_{i_1},\ldots,\Delta_{i_p}\Delta_{n}$ is a {\it complete graph on $(p+1)$ nodes}. We denote by $K_{p+1}$ the clique. 
We now show that $G$ is triangulated. By the induction hypothesis, it suffices to show that cycles of length greater than $3$ containing $\Delta_{n-2}$ have a chord. To this end, observe that if $\Delta_{n-2}$ is in a cycle of length greater than $3$, then $\Delta_{n-2}$ has 2 distinct neighbors  in the cycle. Denote them by $\Delta_{i_j}$ and $\Delta_{i_k}$. Then, necessarily,  both of them belong to $K_{p+1}$. Hence, the edge  $(\Delta_{i_j},\Delta_{i_k})$ is a chord of the cycle. This completes the proof.
\end{proof}

In the sequel, we will require an RHC that yields a given  TLG  $G$ with a particular initialization. We thus show the following:

\begin{Proposition}\label{prop:hennebergfromany}
Let $G$ be an TLG on $n$ nodes with triangles $\Delta_i, 1 \leq i \leq n-2$. Then, for any $\Delta_i$, there exists an RHC starting with $\Delta_i$ that yields $G$.
\end{Proposition}

\begin{proof}
The proof will be carried out by induction on the number of nodes in $G$. The base case of $n=3$ is trivially true. We thus assume that the result holds for any TLG on $(n-1)$ nodes and prove that it holds for TLGs on $n$ nodes.  

Let $G$ be an TLG on $n$ nodes with $(n-2)$ triangles. Then, there is an RHC that builds $G$ by Theorem~\ref{th:equihennebergrtl}.  Without loss of generality, we let $\Delta_{n-2}= \{v_1, v_2, v_{n}\}$ (resp. $v_n)$ be the last triangle (resp. node) appearing in the RHC. Then, the degree of $v_{n}$ is $2$ and $(v_1, v_2)$ is a common edge shared by $\Delta_{n-2}$ with at least one another triangle, say $\Delta_{j} = \{v_1, v_2, v_{k}\}$ for some $k \leq n-1$.

Now, let $G'$ be the subgraph of $G$ induced by the nodes $v_1,\ldots, v_{n-1}$. Then, $G'$ is constructed by stopping an RHC construction after $(n-3)$ steps, and is thus a TLG graph on $(n-1)$ nodes.  By the induction hypothesis, for each triangle $\Delta_i \subset G'$, there exists an RHC starting with $\Delta_i$  that produces $G'$. Note that $\Delta_i$ is a also a triangle of $G$. Continuing  the above RHC by one step joining node $v_{n}$ to nodes $v_1$ and $v_2$ yields an RHC that builds $G$.   

It remains to show that there is an RHC that produces $G$ starting with triangle $\Delta_{n-2}$. This a two-step construction: First, starting from $\Delta_{n-2}$, we add node $v_k$ and connect it to nodes $v_1$ and $v_2$, thus obtaining a graph with $4$ nodes and $2$ triangles ($\Delta_{n-2}$ and $\Delta_j$). This graph is clearly a TLG. For the second step, we appeal again to the induction hypothesis, to obtain an RHC that builds $G'$ starting with $\Delta_{j}$. Since the  concatenation of two RHCs is an RHC, using the two steps above, we have obtained an RHC that produces $G$ from $\Delta_{n-2}$.
\end{proof}

\begin{figure}[h]
    \centering
\includegraphics{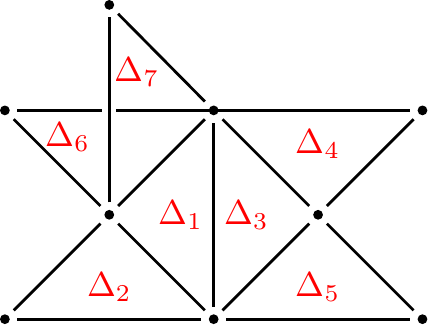}
\caption{the TLG from Fig.~\ref{fig:rtl} is reproduced and the are triangles labeled in the order of appearance with respect to a certain RHC.}
    \label{fig:constructxoxox}
\end{figure}

\begin{Example}
\normalfont
We illustrate here Proposition~\ref{prop:hennebergfromany}. 
Consider the TLG of Fig.~\ref{fig:constructxoxox}. From  Theorem~\ref{th:equihennebergrtl}, we know that there exists an RHC producing it. We label the triangles in the order of appearance with respect to the RHC. The proposition says that one can find an RHC yielding the same $G$ starting from {\it any} $\Delta_i$. Starting from $\Delta_4$, a valid RHC is, e.g.,  $\Delta_4\Delta_3\Delta_5\Delta_1\Delta_2\Delta_6\Delta_7$. 
\end{Example}

The next few propositions shed more light on the  structure of the derived graph $D_G$. 
In particular, both the triangulated and Laman character of $G$ will come into play to show the existence of so-called bottleneck nodes in $D_G$ (see Definition~\ref{def:bottlen} below). These bottleneck nodes will in turn be  essential ingredients in obtaining the limits of the products of local stochastic matrices.

\begin{Proposition}\label{prop:cycleshareedge}
Let $\Delta_1\Delta_2\cdots \Delta_p\Delta_1$ be a cycle in $D_G$ of length greater than~$3$. Then,  
all of these triangles in $G$ share a common edge. In particular, the subgraph of $D_G$ induced by nodes $\Delta_1,\ldots,\Delta_p$ is a complete graph. 
\end{Proposition}

\begin{proof}
The proof is carried out by induction on the length $p$ of the cycle. 

For the base case $p = 3$, we 
first note that two distinct triangles can share at most one edge. 
Assume, without loss of generality, that $\Delta_{i}=\{v_1,v_2,v_3\}$ and $\Delta_j = \{v_1, v_2, v_4\}$, i.e., $(v_1, v_2)$ is the edge shared by $\Delta_i$ and $\Delta_j$.  
If the same edge $(v_1, v_2)$ is also shared by $\Delta_k$, then we are done. Suppose not, say $\Delta_i$ and $\Delta_k$ share edge $(v_1, v_3)$; then $\Delta_j$ and $\Delta_k$ must share edge $(v_1, v_4)$ (it cannot be $(v_2,v_4)$ because otherwise, $\Delta_k$ has four distinct nodes $v_1,\ldots, v_4$.) But, then, the subgraph $G'$ of $G$ induced by $v_1,\ldots, v_4$ is $K_4$. The total number of edges in $K_4$ is $6$, which violates the Laman condition, which states that the number of edges of any induced subgraph on $k$ nodes does not exceed $(2k - 3)$. This proves the base case. See Fig.~\ref{sfig:characterizingDG1} and Fig.~\ref{sfig:characterizingDG2} for illustration.

\begin{figure}[h]
    \centering
    \subfloat[\label{sfig:characterizingDG1}]{
\includegraphics{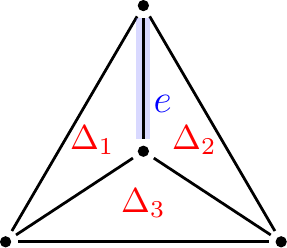}
}\quad 
\subfloat[\label{sfig:characterizingDG2}]{
\includegraphics{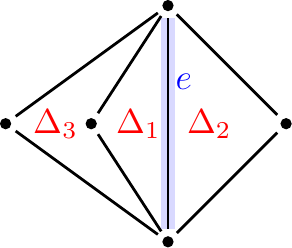}
}\hspace{.3cm}
\subfloat[\label{sfig:characterizingDG3}]{\includegraphics{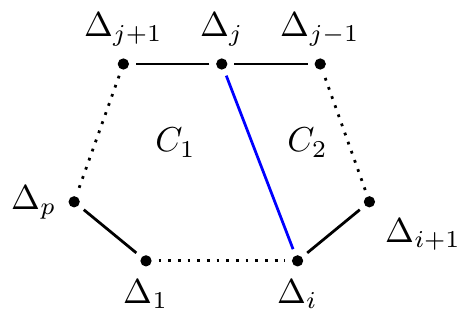}
}
\caption{The three triangles $\Delta_1$, $\Delta_2$, and $\Delta_3$ in~\ref{sfig:characterizingDG1} and~\ref{sfig:characterizingDG2} are pairwise adjacent, so the corresponding derived graphs are triangles. However, the graph in~\ref{sfig:characterizingDG1} cannot be a subgraph of a TLG because it violates the Laman condition for minimal rigidity. Fig.~\ref{sfig:characterizingDG3} illustrates the two cycles $C_1$ and $C_2$ introduced in~\eqref{eq:c1c2}.}\label{fig:characterizingDG}
\end{figure}

For the inductive step we assume that the statement holds for any $p'\le p - 1$ and prove for $p$. Since $p\ge 4$, by Prop.~\ref{prop:tl}, there is a chord $(\Delta_{i}, \Delta_j)$, with $1\le i < j \le p$, in the cycle. Using this chord, we obtain  the following two cycles: 
\begin{equation}\label{eq:c1c2}
\begin{array}{rl}
C_1&:=\Delta_1\cdots\Delta_i\Delta_j\Delta_{j + 1} \cdots\Delta_p\Delta_1\\
C_2&:=\Delta_{i}\Delta_{i + 1} \cdots \Delta_{j-1}\Delta_j \Delta_i 
\end{array}
\end{equation}
of lengths  strictly less than $p$. See Fig.~\ref{sfig:characterizingDG3} for illustration. 
By the induction hypothesis, the triangles in each cycle $C_k$, for $k = 1, 2$, share a common edge $e_k$. Furthermore, note that nodes $\Delta_i$ and $\Delta_j$ appear in both $C_1$ and $C_2$ and, hence, $e_1$ and $e_2$ are shared by both $\Delta_i$ and $\Delta_j$. If $e_1$ and $e_2$ are distinct, then $\Delta_i=\Delta_j$, which is a contradiction. We thus conclude that $e_1 = e_2 = : e$, i.e., the common edge $e$ is shared by all of the triangles in the original cycle. 
\end{proof}

Let $G$ be a TLG with derived graph $D_G$. For a node $v$ in $G$, we denote by $D_G(v)$ the subgraph of $D_G$ induced by the triangles that contain $v$. See Fig.~\ref{fig:tlgconst2} for illustration.

\begin{figure}[h]
    \centering
    \includegraphics{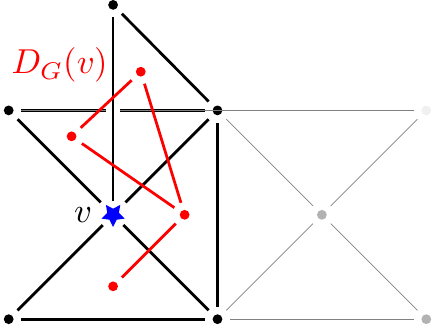}
    \caption{We depict the triangles that contain the starred node $v$ and the corresponding subgraph they induce in $D_G$, denoted by $D_G(v)$. }
\label{fig:tlgconst2}
\end{figure}

\begin{Proposition}\label{prop:dgconnected}
Let $G$ be an arbitrary TLG and $v$ be a node of $G$. Then, $D_G(v)$ is a connected subgraph of $D_G$.
\end{Proposition}

\begin{proof}
We proceed by induction on the number of triangles in $G$ that contain $v$. The base case is such that $v$ belongs to exactly one triangle, say $\Delta_{1}$, in $G$. 
The subgraph of $D_G$ induced by $\Delta_{1}$ is a single node and thus connected. This proves the base case.

For the inductive step, we assume that the statement holds for $(k-1)$ and prove it for $k$. Choose an RHC that builds $G$. Let $\ell$ (resp. $m$) be the step such that the RHC stopped right after step $\ell$ (resp. step $m$) yields a subgraph $G_{\ell} \subset G$ with exactly $(k - 1)$ triangles containing $v$ (resp. a subgraph $G_m \subset G$ with $k$ triangles containing $v$). Since $G_\ell$ is a TLG, by the induction hypothesis, $D_{G_\ell}(v)$ is a connected graph on $(k - 1)$ nodes. Label these nodes as $\Delta_{j_1},\ldots, \Delta_{j_{k-1}}$.   
At step $m$, the RHC chooses an existing edge  $(v, v')\in G_{m-1}$ and adds a node to form a new triangle $\Delta_{j_k}$ that contains $v$. 
Without loss of generality, we assume that $\Delta_{j_1}$ is another triangle that contains the edge $(v,v')$. As a consequence,  $(\Delta_{j_1},\Delta_{j_k})$ is an edge in $D_{G_m}$ that connects $\Delta_{j_k}$ with $D_{G_\ell}(v)$. 
In other words, the subgraph $D_{G_m}(v)$ is connected. Finally, observe that  $D_{G_m}(v)$ and $D_G(v)$ have the same node set by assumption. Since the RHC does not remove existing nodes or edges out of $G$ (and, hence, $D_G$ as well) along the construction process, $D_{G}(v)$ is  connected as well.   
\end{proof}

We now introduce the notion of bottleneck nodes, see Fig.~\ref{fig:squirrel} for an illustration.  

\begin{Definition}\label{def:bottlen}
Let $D$ be an undirected graph, $\alpha$ be a node of $D$, and $D'$ be a subgraph of $D$. A node $\alpha^*\in D'$ is a {\em bottleneck in $D'$ for $\alpha$} if every walk from any node in $D'$ to $\alpha$ contains $\alpha^*$. \end{Definition}

If $\alpha \in D'$, then clearly $\alpha$ is its own bottleneck in $D'$, i.e., $\alpha^*=\alpha$. In most of the time, we are interested in the case where $\alpha\notin D'$. We establish below some relevant properties for bottlenecks. We start with the following one:   

\begin{Lemma}\label{lem:uniquebottlen}
If a bottleneck $\alpha^*\in D'$ for $\alpha$ exists, then it is unique. 
\end{Lemma}

\begin{proof}
The proof is carried out by contradiction. Suppose that there exist two distinct bottlenecks $\alpha^*_1$ and $\alpha^*_2$ in $D'$ for $\alpha$. Let $\gamma = \gamma_1\ldots \gamma_p$ be an arbitrary finite walk with starting node $\gamma_1 = \alpha^*_1$ and $\gamma_p = \alpha$. Since $\alpha^*_2$ is a bottleneck distinct from $\alpha_1^*$, there exists $k_1>1$ such that  $\gamma_{k_1}= \alpha^*_2$.
But, then, $\gamma' := \gamma_{k_1}\cdots \gamma_p$ is a walk from $\alpha^*_2$ to $\alpha$. Note that $|\gamma'| < p$. Similarly, since $\alpha^*_1$ is a bottleneck, there exist another integer $k_2$, with $k_2 > k_1$, such that $\gamma_{k_2} = \alpha^*_1$. Define $\gamma'':=\gamma_{k_2}\cdots \gamma_p$, which is a walk from $\alpha^*_1$ to $\alpha$. By repeatedly applying the above arguments, we obtain an infinite integer sequence $k_1 < k_2 < k_3< \cdots$, such that $\gamma_{k_{2i + 1}} = \alpha^*_{2}$ and $\gamma_{k_{2i}} = \alpha^*_1$. However, the original walk $\gamma$ has finite length, which is a contradiction. We thus have to conclude that $\alpha_1^*=\alpha_2^*$.
\end{proof}

\begin{figure}
    \centering
    \includegraphics{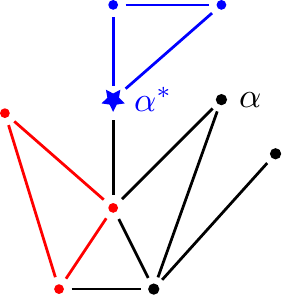}
    
    \caption{Let $\alpha$ be a node in the graph $D$ depicted above. The subgraph of $D$ depicted in blue has a bottleneck for $\alpha$, denoted by $\alpha^*$. The subgraph depicted in red does not have a bottleneck node for $\alpha$.}
    \label{fig:squirrel}
\end{figure}

It should be clear that given the subgraph $D'$ and the node $\alpha$,  if a bottleneck exists, then it is unique. 
The next proposition shows the existence of bottleneck nodes for {\it any} subgraph $D_G(v)$ of $D_G$ and any node outside $D_G(v)$: 

\begin{Proposition}\label{prop:uniqueentry}
Let $\Delta_0$ be a triangle and $v$ be a node in $G$. Let $D_G(v)$ be the subgraph of $D_G$ induced by triangles that contain $v$ in $G$. Then, there exists a bottleneck $\Delta^*\in D_G(v)$ for $\Delta_0$. 
\end{Proposition}

\begin{proof}

If $D_G(v)$ contains one node $\Delta_0$ or if $v \in \Delta_0$, then clearly $\Delta_0$ is the bottleneck. Hence we assume it contains at least two nodes and, moreover, $v\notin \Delta_0$, so $\Delta_0\notin D_G(v)$.  

The remainder of the proof is carried out by contradiction. 
Suppose that there is no bottleneck. By  Prop.~\ref{prop:tl}, one can find two paths, $\gamma$ and $\gamma'$,  that start with nodes in $D_G(v)$ and end at $\Delta_0$. Moreover, by our assumption, $\gamma$ and $\gamma'$ can be chosen with the property that they exit the subgraph $D_G(v)$ through two distinct nodes.

More precisely, let $\gamma_p$ (resp. $\gamma'_p$) the $p$th node in $\gamma$ (resp. $\gamma'$). A node $\gamma_p$ is called the exiting node of $\gamma$ if $\gamma_p\in D_G(v)$ and $\gamma_q\notin D_G(v)$ for any $q > p$. 
Similarly, we let $\gamma'_{p'}$ be the exiting node of $\gamma'$. Then, by the hypothesis, we can find $\gamma$ and $\gamma'$ such that the two exiting nodes $\gamma_p$ and $\gamma'_{p'}$ are distinct. 
For convenience, we assume, by truncating the two paths (if necessary), that the first nodes $\gamma_1$ and $\gamma'_1$ of the two paths are the existing nodes.

We will now construct a cycle in $D_G$ that  contains nodes $\gamma_1$, $\gamma'_1$, and at least one node  not in $D_G(v)$. To this end, since $D_G(v)$ is connected by Prop.~\ref{prop:dgconnected},  there exists a path $\omega$ in $D_G(v)$ from $\gamma_1$ to $\gamma'_1$. 

Next, we let $\Delta$ be the first node that belongs to both $\gamma$ and $\gamma'$. Since $\gamma$ and $\gamma'$ have the same ending node $\Delta_0$, the node $\Delta$ always exists (and it could be $\Delta_0$). 
By concatenating the subpath $\gamma_1\cdots\Delta$ of $\gamma$ with the subpath $\Delta\cdots \gamma'_1$ of $(\gamma')^{-1}$, we obtain a new path $\omega'$ joining $\gamma_1$ to $\gamma_1'$. 

Note, in particular, that only the starting and the ending nodes of $\omega'$ belong to $D_G(v)$. 
By concatenating $\omega$ with $\omega'^{-1}$, we obtain the desired cycle. See Fig.~\ref{fig:<3<3<3} for illustration. 

\begin{figure}
    \centering
    \includegraphics{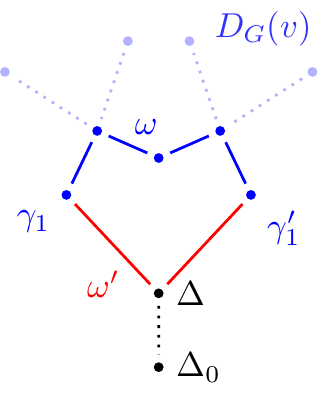}

    \caption{This figure shows that if the bottleneck does not exist, then there would be a cycle $C$ formed by the conjunction of two paths $\omega$ and $\omega'^{-1}$ as shown in the proof. The nodes in blue are nodes of $D_G(v)$. The node $\Delta$ is outside $D_G(v)$ but belongs to $C$, which leads to a contradiction.}
    \label{fig:<3<3<3}
\end{figure}

Denote the cycle by $C$. By Prop.~\ref{prop:cycleshareedge}, the triangles of $G$ that correspond to the nodes of $C$ share a common edge, which we denote by $e$. Because both $\gamma_1$ and $\gamma'_1$ belong to $D_G(v)$ and because $\gamma_1\neq \gamma'_1$, the edge $e$ must contain the node $v$. Since the triangle $\Delta$ is also a node  of $C$, it contains the edge $e$ and, hence, the node $v$. On the other hand, $\Delta$ does not belong to $D_G(v)$, which is a contradiction.   
\end{proof}

\section{Main Results}\label{sec:mainresults}
In the section, we state three main results concerning products of local stochastic matrices (which will be introduced below), namely Theorems~\ref{th:main1},~\ref{th:characwi}, and~\ref{th:surj}. 
Relying on properties of TLGs and their derived graphs established in the previous section, we (1) characterize the limits of these products, (2) make connections between these limits and the so-called absolute probability vectors~\cite{kolmogoroff1936theorie}, 
and (3) show that for any given target limit, one can find a set of local stochastic matrices so that their infinite products converge to the target one.  

In the sequel, we will view $D_G$ as a {\it directed graph} by replacing an undirected edge $(v_i, v_j)$ with two directed ones, namely $v_iv_j$ and $v_jv_i$. 
The purpose of doing so is to emphasize the direction in which an edge  of $D_G$ is travelled.

\subsection{Local stochastic matrices and their infinite products}\label{ssec:lsm}

\paragraph{Local stochastic matrices:} 
We now show how to attach a set of stochastic matrices $A_i$ to a given TLG, and how the same graph can be used to generate an infinite family of products of $A_i$ with a known limit distribution. Throughout this section, $G$ is a TLG on $n$ nodes.

By Proposition~\ref{prop:tl}, there are $(n - 2)$ triangles in $G$, we denote them by $\Delta_1,\ldots, \Delta_{n-2}$. To each $\Delta_i = \{v_j, v_k, v_l\}$, with $j < k < l$, we assign weights to its three nodes.  We denote these weights as $a_{i,j}$, $a_{i,k}$, and $a_{i,l}$ respectively.
We emphasize that a given node does not necessarily have a unique weight assigned, but has one weight assigned per triangle to which it belongs.
On occasion, we will use $a_{\Delta_i,v_j}$, instead of $a_{i,j}$, to denote the weight assigned to node $v_j$ in triangle $\Delta_i$. 

We call $a_i = (a_{i,j}, a_{i,k}, a_{i,l})$ the {\bf local weight vector} associated with $\Delta_i$. 
We next define for each $\Delta_i$ a stochastic matrix as follows:
\begin{equation}\label{eq:defAi}
A_i:=\sum_{v_j, v_k \in \Delta_i} a_{i,k} e_j  e_k^\top + \sum_{v_j \notin \Delta_i} e_je_j^\top.
\end{equation}
The structure of the $A_i$ is easy to state in words: The principal submatrix of $A_i$ corresponding to $\Delta_i$ is a $3$-by-$3$ rank-one stochastic matrix while the remainder of the matrix is simply the identity matrix. We illustrate this below a simple example:

\begin{Example}[Local stochastic matrices with $n=5$]\label{ex:n5partuno}
Consider a graph on $n=5$ nodes consisting of the three triangles $\Delta_1=\{v_1, v_2, v_3\}$, $\Delta_2=\{v_1, v_2, v_4\}$, and $\Delta_3=\{v_2,v_3,v_5\}$. 
In this case, we have the following local stochastic matrices: 
\begin{equation*}
A_1=\begin{bmatrix} a_{1,1} & a_{1,2} & a_{1,3} & 0& 0 \\a_{1,1} & a_{1,2} & a_{1,3} & 0& 0\\a_{1,1} & a_{1,2} & a_{1,3} & 0& 0\\0&0&0&1& 0\\0&0&0&0& 1
\end{bmatrix}, \quad
    A_2=\begin{bmatrix} a_{2,1} & a_{2,2} & 0 & a_{2,4} & 0 \\a_{2,1} & a_{2,2} &0  & a_{2,4} & 0\\0 & 0 & 1 & 0 & 0 \\a_{2,1} & a_{2,2} &0  & a_{2,4} & 0 \\ 0 & 0 & 0 & 0 & 1
    \end{bmatrix},
\end{equation*}
and
\begin{equation*}
    A_3=\begin{bmatrix} 1 & 0 & 0 & 0 & 0 \\ 0 & a_{3,2} & a_{3,3} & 0 & a_{3,5} \\  0 & a_{3,2} & a_{3,3} & 0 & a_{3,5} \\0 & 0 & 0 & 1 & 0 \\ 0 & a_{3,2} & a_{3,3} & 0 & a_{3,5}\end{bmatrix}.
\end{equation*}
\end{Example}

For a later purpose, we need a mild assumption on the local stochastic matrices:

\begin{Assumption}\label{ass:weightnonzero}
For each triangle $\Delta_i$ and each {\it nonsimple} edge $(v_j, v_k)$ in $\Delta_i$, $a_{i,j} + a_{i,k} > 0$.
\end{Assumption}

\paragraph{Products of local stochastic matrices:} 
We now describe allowable products of local stochastic matrices. Let $\gamma$ be a walk in $D_G$. We say that the walk $\gamma$ is {\em infinite} if $|\gamma| =\infty$. To any walk $\gamma:=\Delta_{i_1} \Delta_{i_2} \cdots \Delta_{i_k}$ in the derived graph $D_G$, we associate the product of stochastic matrices 
\begin{equation}\label{eq:product}
P_\gamma:= A_{i_k} \cdots A_{i_2} A_{i_1}.
\end{equation}
We will mostly be interested in the case of infinite walks and, in particular, determining the corresponding product $P_\gamma$.  

The problem, which is twofold in nature, is well-known to be difficult. First, one has to guarantee that the infinite products exists (i.e., in the limit  $|\gamma|\to \infty$).  Second, provided that the limit exists, it usually depends on the complete sequence $\gamma$, making its characterization generically intractable. While the first problem has been been the subject of many investigations, as mentioned in the Introduction, the second problem has received much less attention so far. 

Surprisingly, under certain mild assumptions on the infinite walks (which we introduce in Def.~\ref{def:exhaustive}), a complete characterization of $P_\gamma$ can be obtained. We state the results below. 
To proceed, we first introduce the following definition: 

\begin{Definition}[Exhaustive walk]\label{def:exhaustive}
 A finite walk $\gamma$ in $D_G$ is {\bf exhaustive} if it visits every node of $D_G$ at least once. An infinite walk $\gamma$ in $D_G$ is {\bf exhaustive} if it visit every node of $D_G$ infinitely often. 
\end{Definition}

With the above definition, we  now state the first main result, which says that $P_\gamma$ exists if the infinite walk $\gamma$ is exhaustive and, moreover, there exists a {\it finite} set of rank one matrices to which the limit $P_\gamma$ can belong.  

\begin{Theorem}\label{th:main1}
Let $G$ be a TLG on $n$ nodes, with triangles $\Delta_1,\ldots, \Delta_{n-2}$.  
Let $\{A_1,\ldots, A_{n-2}\}$ be an arbitrary set of local stochastic matrices that satisfy assumption~\ref{ass:weightnonzero}. 
Then, there exist $(n-2)$ probability vectors $\overline {w}_1,\ldots, \overline {w}_{n-2}$ such that for every infinite exhaustive walk $\gamma$ with starting node $\Delta_i$, $1\leq i \leq n-2$, $$P_\gamma= \bfo \overline{w}^\top_i.$$
\end{Theorem}

We introduce below a few corollaries of Theorem~\ref{th:main1}.

\paragraph{Randomized scheduling} 

An infinite exhaustive walk $\gamma$ in $D_G$ can be obtained easily by periodic extension of a finite exhaustive walk whose starting and ending nodes are adjacent.  It can also be obtained via random walks as we describe below. 
Given a node $\Delta_i \in D_G$, we denote by $N(\Delta_i)$ the set of neighbors of $\Delta_i$ (the in-neighbors and the out-neighbors of $\Delta_i$ are the same). We call $\gamma$ a {\it simple random walk} in $D_G$, if $\gamma$ is an infinite walk and the transition probability $\mathbb{P}(\gamma_{t+1}=\Delta_j\mid \gamma_t = \Delta_i)$ is given by 
$$\mathbb{P}(\gamma_{t+1}=\Delta_j \mid \gamma_t=\Delta_i ) = \left\lbrace \begin{array}{ll} \frac{1}{|N(\Delta_i)|} & \mbox{ if }\Delta_j \in N(\Delta_i), \\
0 & \mbox{ otherwise.}
\end{array}\right.$$
Because $D_G$ is connected, it is well known that a simple random walk visits every node of $D_G$ infinitely often (and, hence, it is infinite exhaustive) with probability 1. The following fact is then an immediate consequence of Theorem~\ref{th:main1}:

\begin{Corollary}\label{cor:randomwalk}
Let $\gamma$ be a simple random walk with starting node $\Delta_i$. 
Then, $P_\gamma = \bfo \overline w^\top_i$ with probability one.   
\end{Corollary}

\paragraph{Connection to absolute probability vectors.} 
Theorem~\ref{th:main1} has a few notable consequences.   
Let $\gamma$ and $P_\gamma$ be as in the theorem's statement. We use the common notation that for a pair $0 \le  s < t$ of positive integers, the partial product corresponding to the indices is
$$
P_\gamma(t:s) = A_{\gamma_{t}}A_{\gamma_{t-2}} \cdots A_{\gamma_{s+1}}.
$$
With the above notation, we can write, e.g., $P_\gamma(t:s)P_\gamma(s:r) = P_\gamma(t:r)$.  

Kolmogorov introduced in~\cite{kolmogoroff1936theorie} the {\bf absolute probability vectors} associated with a product $P_\gamma$ (see also~\cite{blackwell1945finite,doob1989kolmogorov}):

\begin{Definition}[Absolute probability vectors]\label{def:apv}
A sequence of vectors $\{x_s\}^\infty_{s = 0}$ are {\bf absolute probability vectors} {\em ({\bf APVs})} for $P_\gamma$ if every $x_s$ is a probabiltiy vector and if for every pair $(s, t)$ of integers, with $0 \le s < t$,  
$x^\top_t P_\gamma(t:s) = x^\top_s$.   
\end{Definition}

APVs are tightly related to the existence of the limit  $\lim_{t \to \infty} P_\gamma(t:0)$. 
For example, it is known that the limit exists if and only if 
there is a {\em unique} set of APVs for $P_\gamma$ and, moreover,  
\begin{equation}\label{eq:apvsconvergence}
\lim_{t \to \infty} P_\gamma(t:s) = \bfo x_s^\top,
\end{equation}
for any given $s \ge 0$.  
We refer the reader to the recent work~\cite{touri2012product} for more details on the use of the APVs (note that the author uses ``absolute probability sequence'' instead of APVs).  
As an immediate consequence of Theorem~\ref{th:main1}, we have the following corollary: 
\begin{Corollary}\label{cor:finiteanduniqueAPVs1}
If $\gamma$ is an infinite exhaustive walk,  
then there is a unique sequence of APVs for $P_\gamma$. 
\end{Corollary}

Furthermore, a complete characterization of the values of the APVs can be obtained using Theorem~\ref{th:main1}: 
\begin{Corollary}\label{cor:finiteanduniqueAPVs}
Let $\gamma$ be an arbitrary infinite exhaustive walk and $\{x_s\}_{s = 0}^\infty$ be the unique sequence of APVs for $P_\gamma$. Let $\overline w_1,\ldots, \overline w_{n-2}$ be as in Theorem~\ref{th:main1}.
Then,  the image of the map $s\mapsto x_s$ is $\{\overline{w}_1,\ldots, \overline{w}_{n-2}\}$.
\end{Corollary}

\begin{proof}
For any $s \ge 0$, we consider the sequence $\gamma':=\gamma_{s+1}\gamma_{s+2} \cdots$, i.e., $\gamma'$ is obtained from $\gamma$ by omitting its first $s$ nodes. 
If $\gamma$ is exhaustive, then so is $\gamma'$. Then, by Theorem~\ref{th:main1},  $P_{\gamma'} = \bfo \overline{w}_{\gamma_{s+1}}$. On the other hand, by~\eqref{eq:apvsconvergence}, we have that
$P_{\gamma'} = \bfo x_s^\top$. It follows that $x_s = \overline{w}_{\gamma_{s+1}}$. This shows that the image of $s \mapsto x_s$ has finite cardinality. Finally, because $\gamma$ is exhaustive, for every $\Delta_i$, there exists an $s$ such that $\gamma_s = \Delta_i$.
\end{proof}

\subsection{Characterization of the product limits}\label{ssec:characprodlim}

Theorem~\ref{th:main1} states that $\lim_{t \to \infty}P_\gamma(t:0)$ exists and can only take value in a finite set. We describe this set below.

\paragraph{Unnormalized APVs:}

\begin{figure}
\centering
     \includegraphics{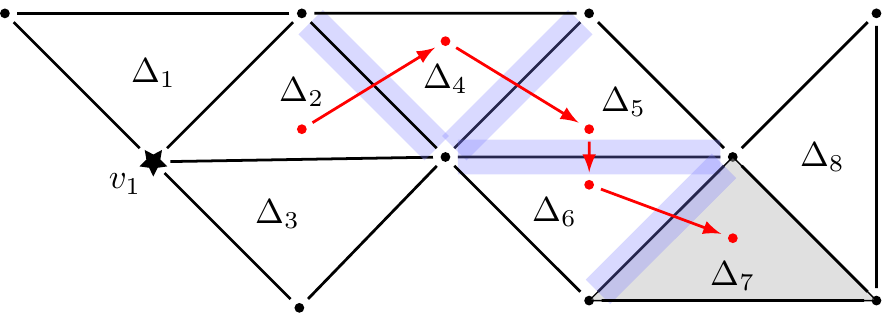}
     \caption{We illustrate the construction of the entry $w_{i,j}$ for the case where $v_j\notin\Delta_i$. The graph $G$ in the figure is a TLG, and the triangles are labeled in order of appearance in a valid RHC starting at $\Delta_1$.  We choose $\Delta_i = \Delta_7$ (the shaded triangle) and $v_j = v_1$ (the star node). There are three triangles $\Delta_{1}$, $\Delta_{2}$, and $\Delta_{3}$ that contain $v_1$, and $\Delta_{2}$ is the bottleneck for $\Delta_7$. The directed path $\gamma = \Delta_2\Delta_4\Delta_5\Delta_6\Delta_7$ is shown in red. The path crosses four edges in $G$, which are highlighted in blue. Each of these is shared by two adjacent triangles along the path $\gamma$. To these edges correspond the four ratios $r_{\gamma_p, \gamma_{p + 1}}$ defined in~\eqref{eq:defr}. The entry $w_{i,j}$ is then the product of these four ratios and the local weight  $a_{\Delta_{2},v_1}$ from $\Delta_2$.}\label{fig:illr}
\end{figure}

For each node $\Delta_i$ in $D_G$, we define a positive vector $w_i \in \R^{n}$ according to the following construction. Let $w_{i,j}$ be the $j$th entry of $w_i$ and recall that $a_i\in\R^3$ is the local weight vector of triangle $\Delta_i$,  introduced at the beginning of Section~\ref{ssec:lsm}.

Let $\Delta_{j_1},\ldots,\Delta_{j_m}$ be the triangles that contain $v_j$. By Prop.~\ref{prop:uniqueentry}, there is a unique $\Delta_{j_k}$, for some $k = 1,\ldots, m$, such that $\Delta_{j_k}$ is the bottleneck in $D_{G}(v_j)$ for $\Delta_i$. 
We have two cases depending on whether $v_j\in\Delta_i$ or $v_j\notin \Delta_i$:
\begin{description}  
    \item[Case 1: $v_j\in \Delta_i$:] In this case, $\Delta_{j_k} = \Delta_i$. We then set $w_{i,j} := a_{i,j}$.

    \item[Case 2: $v_j\notin \Delta_i$:] In this case, we let 
    $\gamma$ be a finite walk in $D_G$ from $\Delta_{j_k}$ to $\Delta_i$.

    Let $\gamma_p$ be the $p$th node in the walk, with  $\gamma_1 = \Delta_{j_k}$. For each $p = 1,\ldots, |\gamma|-1$, we let $e_p = (v_{p_1}, v_{p_2})$ be the {unique} edge in $G$ that is shared by triangles $\gamma_{p}$ and $\gamma_{p + 1}$. Define the ratio 
    \begin{equation}\label{eq:defr}
    r_{\gamma_{p},\gamma_{p + 1}}:= \frac{a_{\gamma_{p+1},v_{p_1}} + a_{\gamma_{p+1}, v_{p_2}}}{a_{\gamma_{p},v_{p_1}} + a_{\gamma_{p}, v_{p_2}}}.\end{equation}
    
    Now, define a product of the above ratios along the path $\gamma$ as follows:
$$
R_{\gamma}:=\prod^{|\gamma|-1}_{p = 1} r_{\gamma_{p},\gamma_{p+1}}.
$$
Apparently, $R_\gamma$ depends on $\gamma$. However, we will show in Prop.~\ref{prop:independenceofpath} that the choice of the walk does not affect $R_\gamma$, as long as the starting and ending nodes are fixed. 

With $R_\gamma$ defined above, we set
    \begin{equation}\label{eq:wijdef}
    w_{i,j} := a_{\gamma_1,j} R_\gamma,
    \end{equation}
    where $a_{\gamma_1, j}$ is the weight of node $v_j$ for $\gamma_1= \Delta_{j_k}$. See Fig.~\ref{fig:illr} for illustration. 
\end{description}

    In words, $r_{\gamma_{p},\gamma_{p+1}}$ is a ratio of the sums of the local weights assigned to the nodes incident to $e_p$:  in the numerator, we take the weights from $\gamma_p$, and in the denominator, the weights from the next triangle in $\gamma$, namely, $\gamma_{p+1}$. We emphasize that the {\it order} of the two subindices of $r$ matters and it reflects the orientation of the edge $\gamma_p \gamma_{p+1}$ in $\gamma$. By Assumption~\ref{ass:weightnonzero}, the denominator in the ratio is strictly positive.  
    
    Note that the above two cases can actually be unified if one extends the definition of $R_\gamma$ to allow for $\gamma$ a path of cardinality $1$ (i.e., a path comprised of a single node). Specifically, we set $R_\gamma = 1$ for any such path. Then, we can express $w_{i,j}$, for $v_j\in \Delta_i$, as $w_{i,j}= a_{i,j}R_{\Delta_i} = a_{i,j}$.

    We now prove the above claim about the {\em independence} of  $R_\gamma$ on the particular walk chosen. It is a consequence of the following fact: 

\begin{Proposition}\label{prop:independenceofpath}
     Let $\omega$ be a closed walk in $D_G$. Then, $R_\omega = 1$. 
\end{Proposition} 

\begin{proof}
We first assume that $\omega$ is a cycle and write $\omega = \omega_1\cdots \omega_k\omega_1$. 
By proposition~\ref{prop:cycleshareedge}, all the triangles $\omega_1,\ldots, \omega_k$ share a common edge in $G$, which we denote by $(v_1, v_2)$. Clearly,  $(v_1, v_2)$ is also the only edge shared by two distinct $\omega_i$ and $\omega_j$. 
It then follows that 
\begin{align*}
R_\omega & = r_{\omega_1,\omega_2}r_{\omega_2,\omega_3}\cdots r_{\omega_k,\omega_1} \\
& = \frac{a_{\omega_2,v_1} + a_{\omega_2,v_2}}{a_{\omega_1,v_1} + a_{\omega_1,v_2}}
\frac{a_{\omega_3,v_1} + a_{\omega_3,v_2}}{a_{\omega_2,v_1} + a_{\omega_2,v_2}}\cdots 
\frac{a_{\omega_1,v_1} + a_{\omega_1,v_2}}{a_{\omega_k,v_1} + a_{\omega_k,v_2}} \\ 
& = 1.
\end{align*}
We now assume that $\omega$ is a closed walk. 
One can always decompose $\omega$ edge-wise and use the directed edges in $\omega$ to form multiple cycles. Label these cycles as $C_1,\ldots, C_m$. Then, it should be clear that     
$R_\omega = R_{C_1}\cdots R_{C_m}$. 
Because $R_{C_i} = 1$ for each $i = 1,\ldots, m$, we have that $R_\omega = 1$. 
\end{proof}
 
Returning to the independence of $R_\gamma$ on a particular path, assume that $\gamma$ and $\gamma'$ are two walks in $D_G$ with the same starting and ending nodes.  Then, by concatenating $\gamma$ with $\gamma'^{-1}$, we obtain a closed walk $\omega$. By Prop.~\ref{prop:independenceofpath}, $R_\omega = R_\gamma R_{\gamma'^{-1}} = 1$. On the other hand, $R_{\gamma'^{-1}} = 1/R_{\gamma'}$. It then follows that $R_\gamma = R_{\gamma'}$, as is claimed above.

\begin{Example}[Unnormalized APVs in case $n=5$]\label{ex:n5partdos}\normalfont
We return to the case  $n=5$ illustrated in Example~\ref{ex:n5partuno}, with three triangles $\Delta_1$, $\Delta_2$, and $\Delta_3$. We denote the associated local weight vectors  $a_1$, $a_2$, and $a_3$. Then, the vectors $w_1$, $w_2$, and $w_3$ obtained according to the above construction are
\begin{align*}
    w_1 &= \begin{pmatrix} a_{1,1} & a_{1,2} & a_{1,3} & a_{2,4}\frac{a_{1,1} + a_{1,2}}{a_{2,1} + a_{2,2}}  &  a_{3,5} \frac{a_{1,2} + a_{1,3}}{a_{3,2} + a_{3,3}} \end{pmatrix}^\top\\
    w_2 &= \begin{pmatrix} a_{2,1} & a_{2,2} & a_{1,3}\frac{a_{2,1} + a_{2,2}}{a_{1,1} + a_{1,2}}   & a_{2,4} & a_{3,5}\frac{a_{1,2} + a_{1,3}}{a_{3,2} + a_{3,3}} \frac{a_{2,1} + a_{2,2}}{a_{1,1} + a_{1,2}}  \end{pmatrix}^\top\\
    w_3 &= \begin{pmatrix} a_{1,1} \frac{a_{3,2} + a_{3,3}}{a_{1,2} + a_{1,3}} & a_{3,2} &a_{3,3}& a_{2,4} \frac{a_{1,1} + a_{1,2}}{a_{2,1} + a_{2,2}}\frac{a_{3,2} + a_{3,3}}{a_{1,2} + a_{1,3}}  & a_{3,5} \end{pmatrix}^\top
\end{align*}\qed
\end{Example}

It should be clear that each vector $w_i$ defined above is nonnegative and nonzero (the entries $w_{i,j}$, for $v_j\in \Delta_i$, defined in case 1 cannot all be zero because they sum to one).  
However, each $w_i$ is not a probability vector because the sum of its entries is in general greater than one.   
The following theorem explains why these vectors are called unnormalized APVs:

 \begin{Theorem}\label{th:characwi}
Let $w_1,\ldots, w_{n-2}$ be the positive vectors defined above. Then, the vectors $\overline w_1,\ldots, \overline w_{n-2}$ in Theorem~\ref{th:main1} are given by normalization: 
$$
    \overline w_i := \frac{w_i}{\sum^n_{j = 1} w_{i,j}}.
$$
 \end{Theorem}

\subsection{Surjectivity of left-eigenvector map}

The remainder of the section concerns the design of stochastic matrices that yield a desired rank one limit for the product $P_\gamma$. We have the following theorem:

\begin{Theorem}\label{th:surj}
Let $\overline{w}^*$ be a positive probability vector. Then, for any given $\Delta_i$, there exist local weight vectors $a_j \in \R^3$, for $ \Delta_j \in D_G$, such that $\overline w_i=\overline w^*$.
\end{Theorem}

The statement is not surprising when one compares the dimensions of the local weight vectors (totaling $2(n-2)$) with the dimension of a probability vector (which is $(n-1)$). 
Nevertheless, the proof we provide is constructive. 
Since the proof relies on a different set of arguments from the ones needed for the previous theorems and since it is relatively simpler,  we provide it here.

\begin{proof}
We proceed by induction on the number of nodes in the graph $G$. 
For the base case $n=3$, we follow Theorem~\ref{th:characwi} and set $a=\overline w^*$.

Now, let us assume that the statement holds for all TLGs $G'$ on $n$ nodes.  
Given a TLG $G$ on $(n + 1)$ nodes  it can be obtained from a TLG graph $G'$ on $n$ nodes by performing one step of RHC.  
By Prop.~\ref{prop:hennebergfromany}, we can assume that the subgraph $G'$ contains $\Delta_i$ (specifically, by Prop.~\ref{prop:hennebergfromany}, we can choose an RHC that starts with $\Delta_i$).  
We denote by $v_{n+1}$ the newly added node going from $G'$ to $G$.

 Let $\overline w^* = (\overline w^*_1,\ldots, \overline w^*_{n+1}) \in \R^{n+1}$ be an arbitrary positive probability vector. 
By the induction hypothesis, we can choose local weights $a_j$, $1 \leq j \leq n-2$, so that the unnormalized vector $(w_{i,1},\ldots,w_{i,n})$ satisfies
$$(w_{i,1},\ldots, w_{i,n}) \propto \frac{(\overline w^*_1,\ldots, \overline w^*_n)}{\sum^n_{i = 1} \overline w^*_i}, $$ i.e., the right hand side is realized as the probability vector for $G'$.  

Let $\gamma$ be a path from  $\Delta_{n + 1}$ to $\Delta_i$.  
Without loss of generality, we can assume that $\Delta_n$ is the second node in the path, so $\Delta_n$ and $\Delta_{n + 1}$ are adjacent,  and we can  label the nodes of $G$ so that $(v_1,v_2)$ is the common edge shared by $\Delta_n$ and $\Delta_{n+1}$. 
Let $\gamma'$ be a subpath of $\gamma$ that starts from $\Delta_n$ and ends at $\Delta_i$.  
Then, from the construction of $w_i$ in Eq.~\eqref{eq:wijdef}, we have that 
\begin{align}
    w_{i,{n + 1}} & = a_{n + 1, n+1} R_\gamma  \nonumber \\ 
    & = a_{n + 1, n+1}  \frac{a_{n, 1} + a_{n,2}}{a_{n+1, 1} + a_{n +1, 2}} R_{\gamma'} \nonumber \\
    & = \frac{a_{n + 1, n+1}}{a_{n+1, 1} + a_{n +1, 2}} (a_{n, 1} + a_{n,2} ) R_{\gamma'}, \label{eq:laststep}
\end{align}
where the second equality follows by unwrapping the definition of $r_{\gamma_1,\gamma_2}$ and the third equality is just a rearrangement. 
Now choose the local weight vector $a_{n+1} = (a_{n+1, 1}, a_{n+1, 2}, a_{n+1,n+1})$ so that the last entry $w_{i, n + 1}$ of $w_i$ satisfies  $$w_{i, n + 1} = \frac{w_{i,n}}{\overline w^*_{n}} \,\overline w^*_{n+1}.$$ 
This can done since $(a_{n,1} + a_{n,2})R_{\gamma'}$ is independent of the local weight vector $a_{n+1}$. 
By normalizing $w_i$, we obtain $\overline w_i = \overline w^*$.
\end{proof}

\section{Proof of Main Results}\label{sec:proofs}
This section is devoted to the proofs of the main theorems stated in Sec.~\ref{sec:mainresults}. 
\subsection{Properties of unnormalized APVs}

We start by deriving some key properties of the unnormalized APVs $w_i$ of the products $P_\gamma$ introduced above. 

\begin{Proposition}\label{prop:wiAiwi}
Let $w_i$ be defined in Sec.~\ref{ssec:characprodlim} and $A_i$ be the corresponding  local stochastic matrix. Then, $w_i^\top A_i = w_i^\top$ for any $i = 1,\ldots, n-2$. 
\end{Proposition}

\begin{proof}
The result is a direct computation. Assume without loss of generality that $\Delta_i$ is comprised of the vertices $v_{1}$, $v_2$, and $v_3$. Then, the matrix $A_i$ takes the form $A_i = \operatorname{diag}(\mathbf{1}_3 a^\top_i ,I_{n-3})$ where $a_i := (a_{i,1}, a_{i,2}, a_{i,3})$. It thus suffices to show that $$(w_{i,1}, w_{i,2}, w_{i,3})^\top \mathbf{1}_3 a^\top_i = (w_{i,1}, w_{i,2}, w_{i,3})^\top,$$
which follows from the fact that $(w_{i,1}, w_{i,2}, w_{i,3}) = a_i$ by  construction of the vector $w_i$. 
\end{proof}

The following Proposition is a major building block in the proofs of the main theorems.

\begin{Proposition}\label{prop:wjAiwi}
Let $\Delta_i$ and $\Delta_j$ be any two adjacent triangles in $G$. 

Then, 
$$
w^\top_j A_i = r_{i,j} w_i^\top,
$$
where $r_{i,j}$ is defined in Eq.~\eqref{eq:defr}.
\end{Proposition}

\begin{proof}
We  prove the proposition by showing that $w_{j,k} = r_{i,j} w_{i,k}$  for all $k = 1,\ldots, n$.  
To do so, we first relabel the vertices (if necessary) so that $\Delta_i$ (resp. $\Delta_j$) is comprised of vertices $v_1$, $v_2$, and $v_3$ (resp. $v_1$, $v_2$, and $v_4$). Then, we have that $A_i = \operatorname{diag}(\mathbf{1}_3 a^\top_i ,I_{n-3})$ with $a_i = (a_{i,1}, a_{i,2}, a_{i,3})$. We consider below two cases for the  subindex $k$: 
\begin{description}
\item[Case 1: $k = 1,2,3$.] In this case, it suffices to show that 
$$(w_{j,1}, w_{j,2}, w_{j,3})^\top \mathbf{1}_3 a_i^\top = r_{i,j} (w_{i,1}, w_{i,2}, w_{i,3})^\top.$$
First, it should be clear that 
$$
(w_{i,1}, w_{i,2}, w_{i,3}) =   (a_{i,1}, a_{i,2}, a_{i,3}).
$$
Next, note that $\Delta_i$ is the bottleneck of $D_G(v_3)$ for $\Delta_j$ because $\Delta_i$ and $\Delta_j$ are adjacent. The construction of $w_j$ described in Sec.~\ref{sec:mainresults} yields that
$$
(w_{j,1}, w_{j,2}, w_{j,3}) = (a_{j,1}, a_{j,2}, a_{i,3}r_{i,j}).
$$
It then follows that
\begin{align*}
(w_{j,1}, w_{j,2}, w_{j,3})^\top \mathbf{1}_3 = a_{j,1} + a_{j,2} + a_{i,3} \frac{a_{j,1} + a_{j,2}}{a_{i,1} + a_{i,2}}  = \frac{a_{j,1} + a_{j,2}}{a_{i,1} + a_{i,2}}= r_{i,j},
\end{align*}
where we have used the fact that $a_{i,1} + a_{i,2} + a_{i,3} = 1$. 

\item[Case 2: $k = 4,\ldots,n$.] 
It suffices to show that $w_{j,4} = r_{i,j} w_{i,4}$ (because the principal submatrix of $A$ formed by the last $(n-3)$ rows/columns is the identity matrix). 
From Prop.~\ref{prop:uniqueentry}, there is a unique bottleneck $\Delta^*$ in  $D_G(v_k)$ for $\Delta_i$. 
Because $\Delta_i$ and $\Delta_j$ are adjacent, the same node $\Delta^*$ is also the bottleneck in $D_G(v_k)$ for $\Delta_j$. 
To see this, let $\gamma$ be an arbitrary walk from a node in $D_G(v_k)$ to $\Delta_j$. Then, $\gamma':= \gamma \vee \Delta_i$ is a walk from the same node in $D_G(v_k)$ to $\Delta_i$. Since $\Delta^*$ is the bottleneck for $\Delta_i$, it is necessarily contained in $\gamma'$ and, hence, $\gamma$. Thus, all walks from $D_G(v_k)$ to $\Delta_j$ contain $\Delta^*$.  Since by lemma~\ref{lem:uniquebottlen} the bottleneck is unique, it has be to $\Delta^*$.

Now, let $\gamma$ be a walk from the bottleneck $\Delta^*$ to $\Delta_j$ and  $\gamma':= \gamma\vee\Delta_i$ as above. Then, $R_{\gamma'} = r_{j,i} R_{\gamma}$. On the other hand, we have $w_{i,k} = a_{\Delta^*, v_k} R_{\gamma'}$ and $w_{j,k} = a_{\Delta^*, v_k} R_{\gamma}$. It follows that $w_{i,k} = r_{j,i}w_{j,k}$. Because $r_{i,j}r_{j,i} = 1$, we have $w_{j,k} = r_{i,j}w_{i,k}$. 
\end{description}
The proof is now complete.
\end{proof}

The above proposition has important implications as we state below: 

\begin{Corollary}\label{cor:closedPratio1}
Let $\gamma$ be a finite walk in $D_G$ that starts at $\Delta_i$ and ends at either $\Delta_i$ or at a node adjacent to $\Delta_i$. 
Then, 
$w^\top_i P_\gamma = w^\top_i$.
\end{Corollary}

\begin{proof}
Let $\Delta_j$ be an arbitrary node {\it adjacent} to $\Delta_i$, and 
$\gamma=\gamma_1\cdots \gamma_{k}$ be a finite walk in $D_G$, with $\gamma_1=\Delta_i$ and $\gamma_k=\Delta_j$.  By adding a node $\Delta_i$ to the end of $\gamma$, one obtains the closed walk $\gamma':=\gamma \vee \Delta_i$. We show that the statement holds for both $\gamma$ and $\gamma'$.

For $\gamma$, we repeatedly apply Prop.~\ref{prop:wjAiwi} to obtain that
\begin{equation}\label{eq:rationonclosedwalk}
\begin{aligned}
w_i^\top P_{\gamma}&=w_i^\top A_{\gamma_k} \cdots A_{\gamma_1} \\
&=r_{\gamma_k, \gamma_1} w_{\gamma_k} A_{\gamma_{k-1}}\cdots A_{\gamma_1}\\
&\, \,\, \vdots \\
& = r_{\gamma_k,\gamma_1}r_{\gamma_{k-1},\gamma_{k}}\cdots r_{\gamma_{1},\gamma_2} w_i^\top \\
& = R_{\gamma'} w_i^\top.
\end{aligned}
\end{equation}
Because $\gamma'$ is a closed walk, $R_{\gamma'} = 1$ by Prop.~\ref{prop:independenceofpath}. 

Next, for $\gamma'$, we note that 
by Prop.~\ref{prop:wiAiwi}, $w_i^\top A_{\gamma_1} =w_i^\top A_i= w_i^\top$, so $$w_i^\top P_{\gamma'} = w_i^\top A_{\gamma_1} A_{\gamma_k} \cdots A_{\gamma_1} = w_i^\top A_{\gamma_k} \cdots A_{\gamma_1} = 
w_i^\top P_{\gamma}.$$
It is shown above that $w_i^\top P_{\gamma} = w_i^\top$. This  completes the proof. 
\end{proof}

\subsection{Exhaustive walks and contraction property}

We develop in this section the necessary tools to show that the limit $P_\gamma$ exists when $\gamma$ is an infinite exhaustive walk in $D_G$. 

\paragraph{Contraction property.} We start by introducing the following semi-norm~\cite{wolfowitz1963products}: Let  $A = [a_{ij}]$ be an arbitrary $n\times m$ matrix. We set
\begin{equation}\label{eq:defsnorm}
\|A\|_S:= \max_{1\le j \le m}\max_{1\le i_1,i_2 \le n} |a_{i_1 j} - a_{i_2 j}|
\end{equation}
It is known~\cite[Theorem 1]{chatterjee1977towards} that if the matrices in the product $P_\gamma$ are stochastic matrices and if 
$\lim_{t\to\infty} \|P_{\gamma}(t:0)\|_S=0$,  then there exists a probability vector   $w$ so that  $\lim_{t\to\infty}P_{\gamma}(t:0) = \bfo_nw^\top$.

We introduce below another known fact: 
\begin{Lemma}\label{lem:ABeps}
Let $A\in \R^{n\times n}$ be a stochastic matrix. 
Suppose that $A$ has a positive column $z$ with $\min z \ge \epsilon > 0$; then, for any matrix $B\in \R^{n\times m}$, 
$$\|AB\|_S < (1 - \epsilon) \|B\|_S.$$  
\end{Lemma}

We do not provide a proof here, but refer to the proof of a similar statement in~\cite[Lemma 3]{hajnal1958weak}.  There, the author also assumes that $B$ is a stochastic matrix. However, the proof provided  does not rely on this assumption. 

\paragraph{Exhaustive walks and products with positive columns.} 
We show  that if $G$ is a TLG  with derived graph $D_G$ and if $\gamma$ is a finite exhaustive walk in $D_G$, then $P_\gamma$ has a positive column. For that, we first have the following fact:

\begin{Lemma}\label{lem:increasingminz}
Let $A$  be a stochastic matrix and $z$ be a nonnegative vector. Then, $\min(Az) \ge \min z$.
\end{Lemma}

\begin{proof}
Since $A$ is row stochastic, every entry of the vector $Az$ is a convex combination of the entries of $z$, and thus larger than $\min z$.
\end{proof}

We now establish the following result, as announced above:

\begin{Proposition}\label{prop:positivecolumn}
There is an $\epsilon\in (0,1)$ such that for any finite, exhaustive walk $\gamma$ in $D_G$, the stochastic matrix $P_\gamma$ has a positive column $z$ with $\min z \ge \epsilon$. 
\end{Proposition}

\begin{proof}
Let $\gamma$ be an arbitrary finite, exhaustive walk. Let $\mu_p$ be the number of distinct nodes  in $
\gamma_1\cdots \gamma_p$. Then $$1 = \mu_1 \le \mu_2 \le \cdots \le \mu_{|\gamma|} = n - 2.$$ 

Because $\mu_p$ is an integer-valued increasing sequence and because 
$\mu_{p+1}-\mu_p \leq 1$ by construction, there exist $(n-2)$ time steps $ t_1<t_2 < \cdots < t_{n-2} := |\gamma|$ with the property that 
$\mu_{t_k + 1}-\mu_{t_{k}}=1$, for $1 \leq k \leq n-3$. Note that $t_1$ has to be $1$ since $D_G$ is simple and $\gamma$ is a walk in $D_G$, 

For ease of analysis, but without loss of generality, we label the nodes of $G$ such that the triangles $\gamma_1,\ldots,\gamma_{t_k}$, for all $k = 1,\ldots,(n-2)$, cover nodes $v_1,\ldots, v_{k+2}$. 
Then, by the definition~\eqref{eq:defAi} of $A_i$, the matrix 
\begin{equation}\label{eq:defP1Pkprod}
P_{\gamma}(t_{k}:0)= A_{\gamma_{t_{k}}} \cdots A_{\gamma_1}=:\operatorname{diag}[Q_k , I], 
\end{equation}
for $k = 1,\ldots, n-2$, 
is block diagonal with $Q_k$ being an $(k+2)\times (k+2)$ stochastic matrix and $I$ being the identity matrix of dimension $(n - k-2)\times (n - k - 2)$.

We show below that for every $k = 1,\ldots, n-2$, there exists an $\epsilon_k > 0$, independent of $\gamma$, such that $Q_k$ has a column $z_k$ with $\min z_k \ge \epsilon_k$. 
The proof is carried out by induction on $k$.  
To proceed, we first define the minimum {\it non-zero} entry over all local weight vectors of triangles: 
\begin{equation}\label{eq:defunderass}
\underline{a}:= \min \left \{ a_{\Delta, v_i}  \mid a_{\Delta, v_i} \neq 0,  v_i \in \Delta, \Delta \in D_G \right \}.
\end{equation}
It should be clear that $\underline{a} \in (0, 1]$.

For the base case $k=1$,  we have that $t_1 = 1$ and $P_\gamma(t_1:0) = A_{\gamma_{1}} =\diag[Q_1, I]$. It follows that $Q_1=(\bfo_3 a_1^\top)=\bfo_3 a_1^\top$, where $a_1 \in \R^3$ is the local weight vector of triangle $\Delta_1$. 
It should be clear that $Q_1$ has a positive column, which we denote by $z_1$. Setting $\epsilon_1 := \underline{a}$, we obtain that $\min z_1 \geq \epsilon_1.$

For the inductive step, we assume that $\epsilon_{k-1}$ exists and prove the existence of $\epsilon_{k}$. 
By~\eqref{eq:defP1Pkprod}, 
we have that $$P_\gamma(t_{k-1}:0) = \diag[Q_{k-1}, I],$$ where $\dim Q_{k-1}$ is $(k + 1)\times (k + 1)$. 
Now, consider the sub-walk $\gamma_1\cdots \gamma_{t_{k}}$ and the corresponding product $P_\gamma(t_{k}:0)$. We can express $P_\gamma(t_{k-1} + 1:0) = \diag[Q'_{k }, I]$, where $Q'_{k}$ is of dimension $(k+2)\times (k + 2)$. 

By an earlier assumption, triangle $\gamma_{t_{k-1} +1}$ contains the node $v_{k + 2}$. 
Moreover, by relabeling nodes $v_1,\ldots,v_{k+1}$, we can arrange matters so that $\gamma_{t_{k-1} + 1}$ has nodes $\{v_{k},v_{k+1},v_{k+2}\}$. 
Denote by $\tilde Q_{k-1}$ the matrix $Q_{k-1}$ after the relabeling, then  $\tilde Q_{k-1}$ is obtained by a permutation of rows/columns of $Q_{k-1}$. It should be  clear that if $Q_{k-1}$ has a column $z_{k-1}$ such that $\min z_{k-1} \ge \epsilon$ for some $\epsilon > 0$, then so does $\tilde Q_{k-1}$. For ease of notation, we will still write $Q_{k-1}$ instead of $\tilde Q_{k-1}$.    
Using this labeling,
we can compute $Q'_{k}$ from $Q_{k-1}$ via the following expression:
$$
Q'_{k} = \diag[I, \bfo_3 a^\top] \diag[Q_{k-1}, 1]. 
$$
where $a\in \R^3$ is the local weight vector corresponding to the triangle $\gamma_{t_{k-1} + 1}$ (the sub-index of $a$ has been omitted for simplicity). 

By the induction hypothesis, there is a positive column $z_{k-1}$ of $Q_{k-1}$ such that $\min z_{k-1} \ge \epsilon_{k-1}$.  
Then, $[z_{k-1}; 0]$ is a column of $\diag[Q_{k-1}, 1]$. Let $z_{{k-1},j}$ (resp. $a_j$) be the $j$th entry of $z_{k-1}$ (resp. $a$). We compute below the column vector $z'_{k }:=\diag[I, \mathbf{1}_3 a^\top] [z_{k-1}; 0]$: 
\begin{equation}\label{eq:defzkp1}z'_{k,j} =\left\lbrace \begin{array}{ll}
     z_{{k-1},j}& \mbox{ for } 1 \leq j \leq k - 1,  \\
     a_1 z_{k-1,k}+a_2 z_{{k-1},k+1}& \mbox{ for } k \leq j\leq k+2. 
\end{array}
\right.
\end{equation}
From Assumption~\ref{ass:weightnonzero}, $a_{1}$ and $a_2$ cannot both be zero, so  $Q'_{k}$ has a positive column $z'_{k }$. Moreover, because $\min z_{k-1} \ge \epsilon_{k-1}$ and $a_1, a_2 \ge \underline{a}$ (note that $\underline{a} < 1$), 
$$
\min z'_{k} \ge  \underline{a} \epsilon_{k-1} = : \epsilon'_{k} \in (0, 1). 
$$
Note that $\epsilon'_{k}$ is independent of $\gamma$ and, hence, a uniform lower bound.

By the definition of $t_{k}$,  the sub-walk $\gamma_1 \cdots \gamma_{t_{k}}$ covers the same set of nodes as the sub-walk $\gamma_1 \cdots \gamma_{t_{k-1} + 1}$ does. In particular, using Eq.~\eqref{eq:defP1Pkprod}, the  matrix $P_\gamma(t_{k} :0)$ takes the form $\diag[Q_{k}, I]$ with $Q_{k}$ and $Q'_{k}$ of the same dimension. Moreover, $Q_{k}$ can be expressed as $Q_{k} = S_{k} Q'_{k}$ where $S_{k}$ is the  $(k+2)\times (k+2)$ leading principal submatrix of $P_\gamma(t_{k}: t_{k-1} + 1)$. Note that $S_{k}$ is a stochastic matrix. By Lemma~\ref{lem:increasingminz}, $\min (S_{k} z'_{k}) \ge \min z'_{k}$, which implies that $Q_{k}$ has a positive column $z_{k}$ and, moreover, $\min z_{k} \ge \epsilon_{k}:=\epsilon'_{k}$.
\end{proof}

\begin{Remark}\normalfont
A close inspection of the above arguments yields that  $\epsilon$ can be chosen to be $\epsilon = \underline{a}^{n - 2}$.   
\end{Remark}

\subsection{Proofs of Main Theorems}\label{ssec:mainproof}
We now put the above result together to prove Theorems~\ref{th:main1} and~\ref{th:characwi}.

Let $\gamma$ be an infinite, exhaustive walk starting at node $\Delta_i$. Let $t_k$, for $k\ge 0$, be a monotonically increasing sequence of time steps 
at which the walk first revisits $\Delta_i$ after having visited all other nodes since $t_{k-1}$, i.e., $\gamma_{t_k}=\Delta_i$ and $\gamma_{t_k}\cdots\gamma_{t_{k + 1} - 1}$ is a finite, exhaustive walk. We set $t_0:=1$. By assumption, $(n - 2) \leq t_{k+1}-t_k< \infty$. For convenience, we introduce $s_k:= t_k - 1$, so $s_0 = 0$.  

From  Lemma~\ref{lem:ABeps}, the semi-norm $\|P_{\gamma}(s:0)\|_S$ is {\it non-increasing} in $s$ and since it is obviously lower-bounded, it converges to a limit. We now show that this limit is $0$. 
To this end, we use  Prop.~\ref{prop:positivecolumn} to obtain an $\epsilon\in (0,1)$ such that for each $k \geq 0$, $P_\gamma(s_{k+1}:s_k )$ has a positive column $z_k$ with $\min z_k \geq \epsilon$. Then, by repeatedly applying Lemma~\ref{lem:ABeps}, we have that
\begin{multline}
\| P_\gamma(s_{k}:s_0) \|_S = \|P_\gamma(s_{k}:s_{k-1}) P_\gamma(s_{k-1} :s_0) \|_S\\ \le (1 - \epsilon) \|P_\gamma(s_{k-1}:s_0 )\|_S \le (1- \epsilon)^{k}.
\end{multline}
It then follows that 
$$\|P_\gamma\|_S = \lim_{k\to \infty} \|P_\gamma(s_k :s_0)\| = \lim_{k\to \infty} (1 - \epsilon)^k = 0,$$ as claimed. 
Using~\cite[Theorem 1]{chatterjee1977towards}, we conclude that $P_\gamma$ converges to a rank-one stochastic matrix.

Let $\hat w$ be such that $P_\gamma=\bfo \hat w^\top$.  
We next show that $\hat w =  \overline w_i$ as is in the statement of Theorem~\ref{th:characwi}. On the one hand, the convergence of $P_\gamma(s_k:s_0)$ to $P_\gamma$ implies the convergence of the row vectors $w_i^\top P_\gamma( s_k:s_0)$ to the row vector $w_i^\top P_\gamma = w_i^\top \bfo \hat w^\top = \hat w^\top$.  
On the other hand, each finite walk $\gamma_{t_0}\cdots \gamma_{t_{k} - 1}$ starts with $\gamma_i$ and ends with a node {\it adjacent} to $\gamma_i$. Thus, by Corollary~\ref{cor:closedPratio1}, $ w_i^\top P_\gamma(s_{k}:s_0)  = w_i^\top$. Combining the above arguments, we have that 
$$
w_i^\top = \lim_{k\to \infty} w_i^\top P_\gamma(s_{k}:s_0) = w_i^\top P_\gamma = \hat w^\top.
$$
This completes the proofs of both Theorems~\ref{th:main1} and~\ref{th:characwi}. 
\hfill\qed

\section{Numerical studies}\label{sec:numerical}
We present here simulation results showing the validity of the main theorems and the importance of the adjacency rule and the condition that the underlying graph is a TLG. 
Precisely, we present three sets of experiments. In the first one, we show that if $G$ is not a TLG, but simply a triangulated graph, then the conclusions do not hold. 
In the second set, we explore the importance of the structure of the local stochastic matrices. 
In the last set, we show that the 
adjacency rule for the product, i.e., that $\gamma$ is a walk in $D_G$, is critical as well.

\paragraph{Experiment 1: On triangulated Laman graphs.}
We consider in Fig.~\ref{sfig:5pi} a triangulated graph with $5$ triangles. The derived graph is a cycle of length $5$. The graph is not Laman because it violates the Laman condition, but it is rigid.
To each triangle, we assign a local weight vector, which yields the associated local stochastic matrix. These local weight vectors are {\em i.i.d} random variables uniformly drawn from $\sp(2)$. 

We then sample $N_r=50\cdot 10^3$ random walks $\gamma$ in $D_G$. The cardinality of each walk is $N=10^4$. This length was sufficient to guarantee converge of the product $P_{\gamma}$ to a rank-one matrix, as was observed in the simulation (we took the absolute values of the eigenvalues of the products and verified that there was only one nonzero value, namely value one).

We denote by $w$ the left eigenvector of $P_{\gamma}$ corresponding to eigenvalue $1$, i.e., $P_{\gamma}=\bfo (w)^\top$. Then, $w$ is a random variable taking value in $\sp(5)$.  
We plot in Fig.~\ref{fig:histnotlg} the  histograms for the 6 entries of $w$. We observe that the support of these empirical distributions has non-zero measure, indicating that there is a continuum of limits, associated to the chosen set of local weight vectors.

\begin{figure}[h]
\centering
\subfloat[\label{sfig:5pi}]{
    \includegraphics{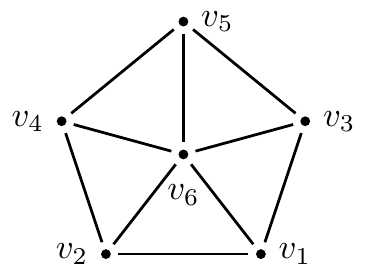}

}
\qquad
\subfloat[\label{sfig:4line}]{
    \includegraphics{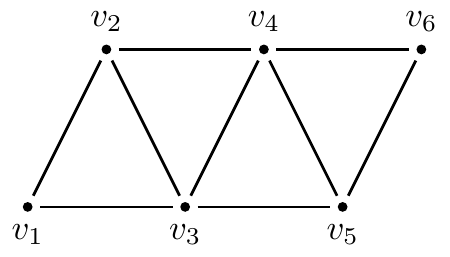}

}
\caption{The graph in Fig.~\ref{sfig:5pi} is a triangulated rigid graph, but is not Laman.  The graph in Fig.~\ref{sfig:4line} is a TLG.
}
\end{figure}

\begin{figure}[h]
    \centering

\includegraphics[width=1\columnwidth]{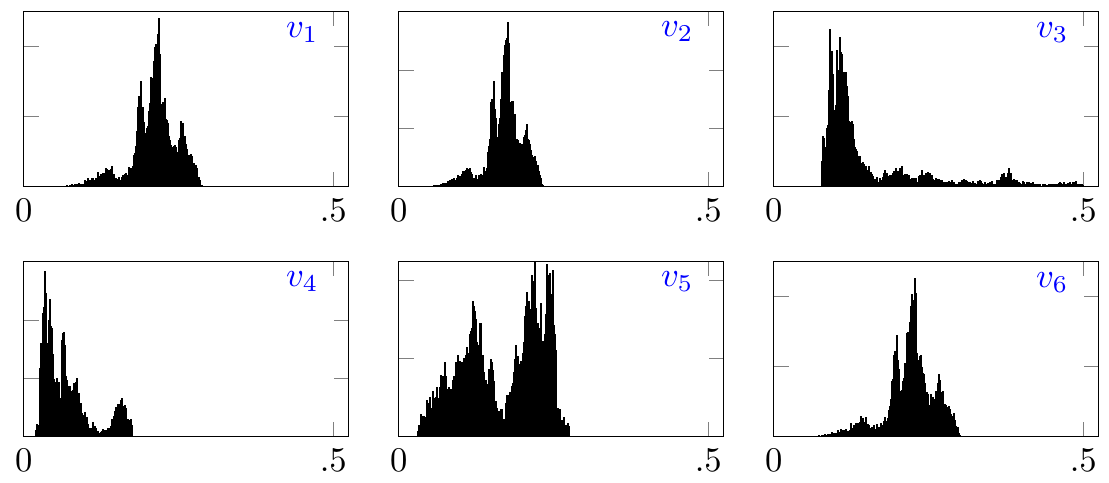}
\caption{For random walks $\gamma$, with $|\gamma| = 10^4$, in the derived graph of Fig.~\ref{sfig:5pi}, we evaluate the corresponding left eigenvector $w$ of $P_{\gamma}$. We plot the empirical distribution for each entry of $w$. } \label{fig:histnotlg}
\end{figure}

\paragraph{Experiment 2: On local stochastic matrices.}

We consider in Fig.~\ref{sfig:4line} a TLG with $4$ triangles whose  derived  graph is line graph.
To each triangle $\Delta_i$, we assign a random $3\times 3$ stochastic matrix, realized as the principal submatrix of $A_i$ corresponding to the nodes of $\Delta_i$. 
All of the row vectors of these $3\times 3$ matrices are   
{\em i.i.d} random variables uniformly drawn from $\sp(2)$.  
This construction of local weight matrices violates our assumption on the $A_i$'s, which requires these principal submatrices to have identical rows. 

Similarly, we sample $N_r=50\cdot 10^3$ random walks $\gamma$ of cardinality $N=10^4$ in $D_G$ (for which we observe convergence of every $P_\gamma$) and plot the empirical distribution of the entries of the limiting left-eigenvector of $P_\gamma$.  We again observe that the support of these empirical distributions have non-zero measures, indicating that there is a continuum of limits.

\begin{figure}[h]
    \centering
\includegraphics[width=1\columnwidth]{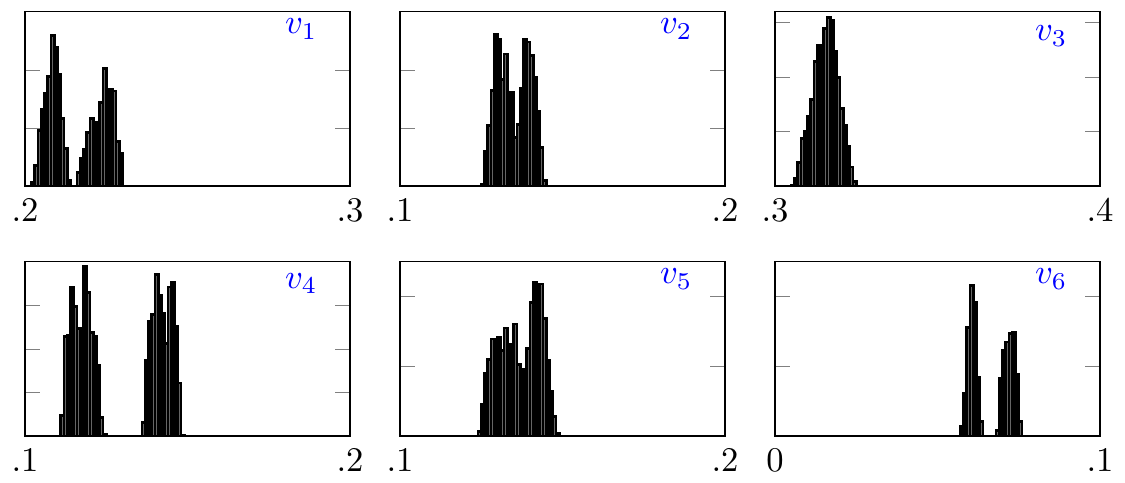}
    \caption{For random walks $\gamma$, with $|\gamma| = 10^4$, in the derived graph of Fig.~\ref{sfig:4line}, we evaluate the corresponding left eigenvector $w$ of $P_{\gamma}$. We plot the empirical distribution for each entry of $w$. }\label{fig:histnolsm}
\end{figure}

\paragraph{Experiment 3: On adjacency rules.}
In this last set of experiments, we verify that if the assumptions are met, $P_\gamma$ converges to a rank-one matrix whose value does not depend on the walk $\gamma$ in $D_G$. We also verify that if all assumptions are met but $\gamma$ is a random sequence of triangles, and thus does not respect the adjacency rules afforded by $D_G$, the conclusions do not hold.
The simulations shown in Fig.~\ref{fig:histhorse} are based on a larger TLG $G$ with 18 nodes, depicted in Fig.~\ref{sfig:horseG}, with derived graph $D_G$ on $16$ nodes shown in Fig.~\ref{sfig:horseDG}.  To each triangle, we assign a local weight vector, which yields the associated local stochastic matrix. These local weight vectors are {\em i.i.d} random variables uniformly drawn from $\sp(2)$. 

We first sampled $N_r=50\cdot 10^3$ random walks $\gamma$ of cardinality $N=30 \cdot 10^3$ in $D_G$. Every walk starts at node $\Delta_1=\{v_1,v_2,v_3\}$. 
We observe convergence of every $P_\gamma$ to a common rank-one matrix. This provides numerical support of Theorem~\ref{th:main1}. 
Denote by $w$ the common left-eigenvector of $P_\gamma$ corresponding to eigenvalue $1$, and $w_i$ its entry. 
We plot the delta functions in Fig.~\ref{fig:histhorse} in red at these $w_i$.

We then sampled $N_r=50\cdot 10^3$ random sequences $\gamma$ of cardinality $N=30 \cdot 10^3$ in $D_G$ (for which we observe again convergence of every $P_\gamma$ to a rank-one matrix). In this case, each element of the sequence is a randomly chosen node in $D_G$.  We plot in Fig.~\ref{fig:histhorse} in black the empirical distributions for the entries of the left-eigenvector of $P_\gamma$ corresponding to the eigenvalue 1. 
Again, observe that the support of these empirical distributions have non-zero measures, indicating that there is a continuum of limits.

\begin{figure}
    \centering
\subfloat[\label{sfig:horseG}]{
    \includegraphics{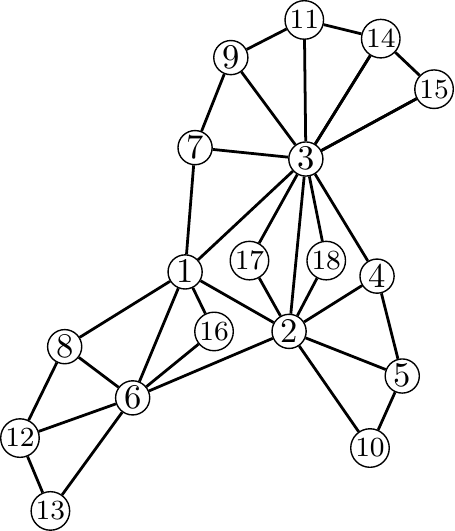}

} \qquad \quad
\subfloat[\label{sfig:horseDG}]{
\includegraphics{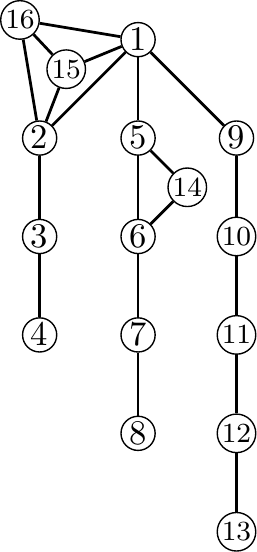}
}
\caption{The ``horse'' graph $G$ in Fig.~\ref{sfig:horseG} is a TLG on 18 nodes. We plot its derived graph $D_G$ in~Fig.~\ref{sfig:horseDG}. We provide here a few (but not all) correspondences between triangles in $G$ and nodes in $D_G$: $\Delta_1 = \{v_1,v_2,v_3\}$, $\Delta_2 = \{v_2,v_3,v_4\}$, $\Delta_5 = \{v_1,v_2,v_6\}$, and $\Delta_9 = \{v_1,v_3,v_7\}$.}
    
\end{figure}

\begin{figure}
   \centering
\includegraphics[width=1\columnwidth]{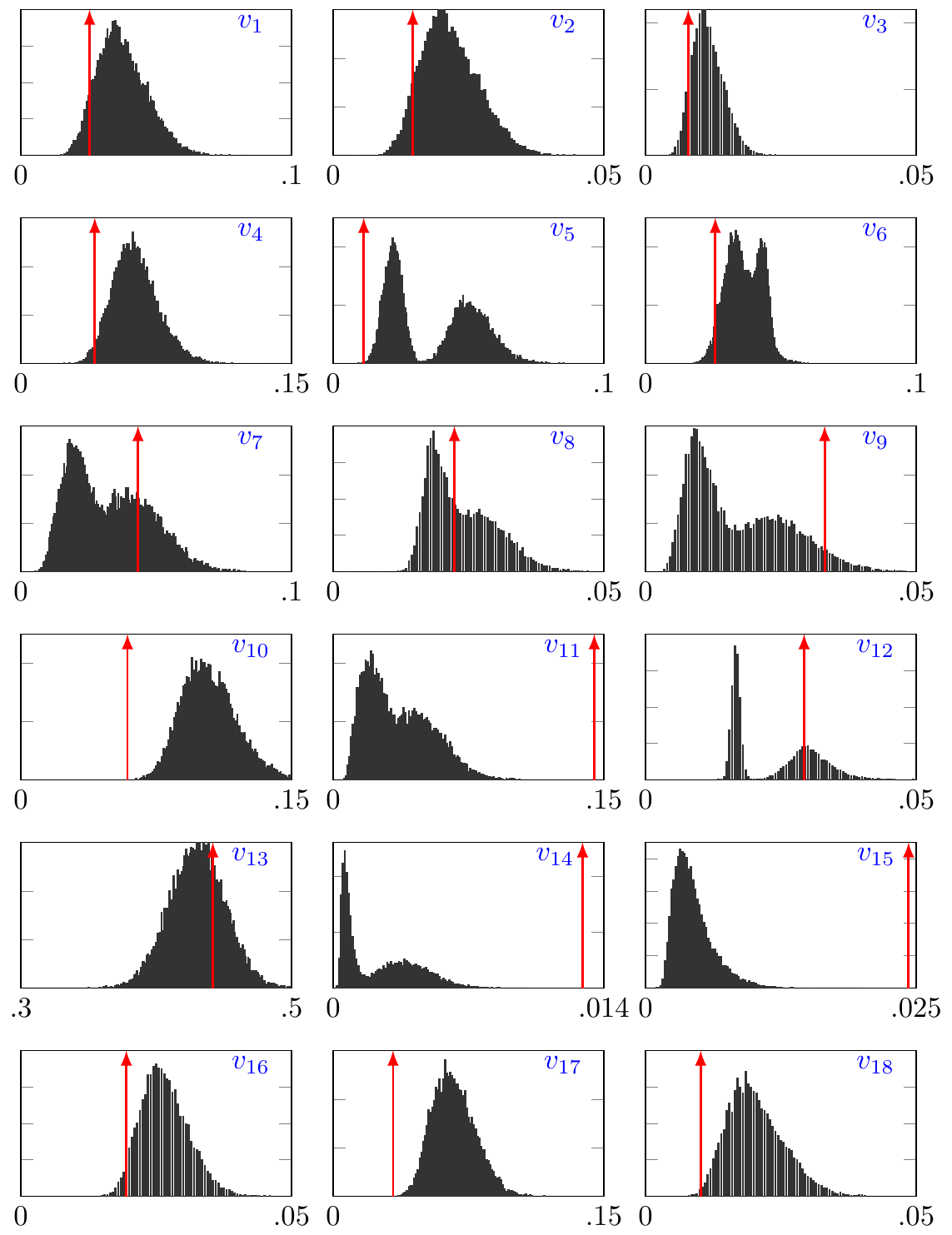}

\caption{For random sequences (not necessarily walks) $\gamma$, with $|\gamma| = 30\cdot 10^3$, of nodes in the derived graph shown in Fig.~\ref{sfig:horseDG}, we evaluate the corresponding left eigenvector $w$ of $P_{\gamma}$. We plot the empirical distribution for each entry of $w$. The delta functions in red are distributions that correspond to random walks $\gamma$ in $D_G$. The distributions in black correspond to sequences $\gamma$ of nodes in $D_G$ chosen uniformly at random.}\label{fig:histhorse}
\end{figure}

\section{Summary and outlook}\label{sec:conclusion}

We have shown how to construct sets of stochastic matrices (called local stochastic matrices) and adjacency rules for taking product of these matrices that guarantee, under some mild assumptions, convergence of the product to one out of a finite number of possible limits. These limits are all  rank-one matrices and, knowing the first matrix in the product is enough to determine which limit the product will converge to.

Underlying our work is the notion of triangulated Laman graph (TLG), where we recall that a Laman graph is a minimally rigid graph. The local stochastic matrices are in one-to-one correspondence with the triangles of this graph, and the adjacency rules for their product are encoded in the hereby defined derived graph of a TLG. 

The connections between the minimal rigidity of the underlying graph and the convergence of the product appear at several point in the proof: first, and foremost, in the construction of the unnormalized APVs $w_i$, where properties of the derived graph of a TLG are integral to the argument; second, in the proof of convergence itself. These connections are either direct, using characterizations of Laman graphs, or rely on the existence of a Restricted Henneberg Construction for the graphs, a fact we proved in the appendix and that requires the graph to be minimally rigid.

We have  provided simulations showing that departing from our assumptions---e.g.,  using a rigid, but non-minimally rigid, graph, or  changing the structure of the local stochastic matrices, or disrespecting the adjacency rules in the products--- would generically break the conclusions of the Theorems, in particular the conclusion about the number of possible limits of the product.

Beyond their application in the present paper, we believe that some of the novel ideas introduced can form the basis of a broader set of results. In particular, one of the key facts for proving the finite cardinality of the set of possible limits is Prop.~\ref{prop:wjAiwi}. There, we have established the relation $w^\top_j A_i = r_{i,j} w_i^\top$, where $i$ and $j$ correspond to {\it adjacent nodes} in the derived graph; indeed, the fact that the limits of the convergent products are exactly $\bfo\overline w^\top_i$ follow as a consequence of the proposition as shown in Sec.~\ref{ssec:mainproof}. 

The above relation motivates us to consider the following problem: Suppose that one is given a finite set of stochastic matrices $\{A_1,\ldots, A_k\}$ and a set of nonnegative vectors $\{w_1,\ldots, w_k\}$.
Define a directed graph $D$ on $k$ nodes as follows: there is an edge from node $i$ to node $j$ if $w^\top_j A_i \propto w_i$. Then, under what circumstances is the graph $D$  strongly connected? If it is, when is every $P_\gamma$  a rank-one matrix for infinite exhaustive walk $\gamma$ in $D$?  If we meet conditions so that the answer to the above questions are positive, then $P_\gamma$ is a rank-one matrix and if the product starts with matrix $A_i$, then the rank-one matrix has to be $\bfo \overline w_i^\top$, where $\overline w_i$ is again the normalized version of $w_i$. The associated sequence of absolute probability vectors take values $\overline w_1,\ldots, \overline w_k$. 

A few simple examples fitting the above framework are the case of all matrices $A_i$ being the same, or all matrices $A_i$ commuting with each other. In both cases, it is easy to see that $D$ is the complete graph and a positive answer to the above questions is obtained whenever the $A_i$'s are irreducible. A more involved example is the one of gossiping: in this case, the matrices $A_i$ are doubly stochastic matrices, one-to-one correspondent to the edges of a connected undirected graph. All the vectors $w_i$ are chosen to be $\bfo$. The directed graph $D$ is again the complete graph. The uniqueness of the limit for any infinitely exhaustive walk in $D$ is a consequence of the doubly-stochastic nature of the $A_i$'s. 
Finally, in the case of the present paper, the $A_i$'s are local stochastic matrices in one-to-one correspondence with triangles of a TLG,  and the directed graph $D$ is the directed version of the derived graph $D_G$. Not much is known beyond these cases. By proposing the framework, and the attendant  questions raised above, we look for solutions that generalize and unify the existing results and the results established in the present paper.

\bibliographystyle{plain}
\bibliography{secureconsensus}

\appendix
\section{Proof of Theorem~\ref{th:equihennebergrtl}}
\paragraph{On rigidity theory.} A graph is rigid if, upon embedding the graph in a Euclidean space $\R^k$, 
fixing all edge lengths precludes motions of the vertices, save for translations and rotations of the embedded graph. A graph is {\em minimally rigid} if it is rigid, and no edge can be removed without losing that property. Rigidity in dimension two (i.e., $k = 2$) is relatively well-understood, but many basic questions remain open in dimensions three and above. Minimally rigid graphs in dimension two are called {\em Laman graphs}. We refer the reader to~\cite{graver1993combinatorial} for formal definitions.

A major result in rigidity theory is the so-called Laman condition, which completely characterizes minimally rigid graphs in dimension two. 

\begin{Lemma}[Laman's condition]\label{lem:lamancondition}
An undirected graph $G=(V,E)$ on $n$ nodes is minimally rigid if and only if:
\begin{enumerate}
    \item There are $(2n - 3)$ edges in $G$;
    \item Every induced subgraph of $G$ on $k$ nodes, for $2\le k \le n-1$, has at most $(2k-3)$ edges.
\end{enumerate}
\end{Lemma}

\paragraph{Henneberg construction.} A basic tool in rigidity theory is the so-called Henneberg construction. It is known that every minimally rigid graph admits a {\it Henneberg construction} and, reciprocally, every Henneberg construction yields a minimally rigid graph. We describe the Henneberg construction below: Starting with an edge, the Henneberg construction iteratively adds a node by applying one of the following two operations at each stage:

\begin{description}
\item[1. Node-add:] Select two nodes $v_i,v_j$ in $G_{n-1}$ , add a node $v_n$ and the edges $(v_i,v_n)$ and $(v_j,v_n)$ to obtain $G_n$.
\item[2. Edge-split:] Select an edge $(v_i,v_j)$ and a node $v_k$ in $G_{n-1}$, add a node $v_n$ and edges  $(v_i,v_n), (v_j,v_n)$ and $(v_k,v_n)$ remove edge $(v_i,v_j)$.
\end{description}

The sequence of graphs obtained following the construction is called a {\em Henneberg sequence}. Each graph in a Henneberg sequence is minimally rigid. 
It should be clear that the RHC introduced in Sec.~\ref{sec:tlg} is a type of Henneberg construction that starts with a triangle.

\paragraph{A Henneberg construction for TLGs} We introduce below a few preliminary results that are needed for the proof of Theorem~\ref{th:equihennebergrtl}.

\begin{Lemma}\label{lem:sufficiency}
Let $G$ be a TLG and  $G_3,\cdots, G_n$ be a sequence of graphs obtained by following the steps of an RHC, with $G_3$ a triangle. Then, every graph in the sequence is a TLG. 
\end{Lemma}

\begin{proof}
Since the RHC is a type of Henneberg construction, each $G_i$ in the sequence is Laman. We show that it is also triangulated. 
We proceed by induction of the number of nodes $n$ in the graph. For $n=3$, $G_3$ is a triangle and the statement holds. Now assume that $G_{n-1}$ is a TLG. Let $v_n$, $(v_n,v_i)$, $(v_n,v_j)$ be the  node and edges newly added to $G_{n-1}$ to obtain $G_n$. Consider any cycle $\gamma$ of length greater than 3 in $G_n$. 
Either $\gamma$ does not contain $v_n$: it is then included in $G_{n-1}$ and since $G_{n-1}$ is triangulated, it contains a chord. Otherwise,  $\gamma$ contains $v_n$: then it necessarily contains $v_i$ and $v_j$ as they are the only two nodes adjacent to $v_n$. Since $(v_i,v_j)$ is an edge in $G_{n-1}$ and thus in $G_n$, the cycle $\gamma$ has a chord, namely $(v_i, v_j)$.  
\end{proof}

In the following Lemma, we show that any {\em Henneberg} construction for a TLG yields a sequence in which each graph is also a TLG. 

\begin{Lemma}\label{lem:tlsubgraphs}
Let $G$ be a TLG. Fix a Henneberg construction for $G$ and denote by $G_3,\cdots, G_n=G$ the Laman graphs obtained in that construction, with $G_3$ a triangle. Then, each $G_i$ is triangulated. 
\end{Lemma}

\begin{proof}
We first show that $G_{n-1}$ is triangulated. 
Assume, by contradiction, that $\gamma=v_1\cdots v_kv_1$, $k \ge 4$, is a chord-free cycle in $G_{n-1}$ of length greater than $3$. On the one hand, if $G_n$ is obtained from $G_{n-1}$ with a node-add operation, then clearly $\gamma$ is also a chord-free cycle in $G_n$. On the other hand, assume that $G_n$ is obtained using an edge-split operation. If the selected edge of the edge-split operation is not part of $\gamma$, then $\gamma$ is a chord-free cycle in $G_n$. We thus assume that the selected edge is part of $\gamma$, say $(v_1,v_2)$, and denote by $v_\ell$ the selected node of $G_{n-1}$. Note that $v_\ell \neq v_1,v_2$. If $v_\ell \notin \gamma$, then $v_1v_nv_2v_3\cdots v_kv_1$ is a chord-free cycle of length $(k+1)$ in $G_n$. If $v_\ell \in \gamma$, then $v_1v_nv_\ell v_{\ell-1}\cdots v_1$ and $v_2v_3\cdots v_\ell v_nv_2$ are two distinct chord-free cycles in $G_n$ (see Fig.~\ref{fig:twocyclesedgesplit} for illustration). The sum of the lengths of the two cycles is $(k+3) \geq 7$, so at least one of them is of length greater than 3. We have thus shown that if $G_{n-1}$ has a chord-free cycle of length greater than 3, then so does $G_n$. It thus contradicts the assumption that $G_n$ is triangulated. 
 
We have just shown that if $G_n$ is triangulated, then so is $G_{n-1}$. 
Applying the above arguments iteratively, we obtain that $G_{n-2},\ldots, G_3$ (in a reversed order) are all TLGs  as announced.  
\end{proof}

\begin{figure}
    \centering
    \includegraphics{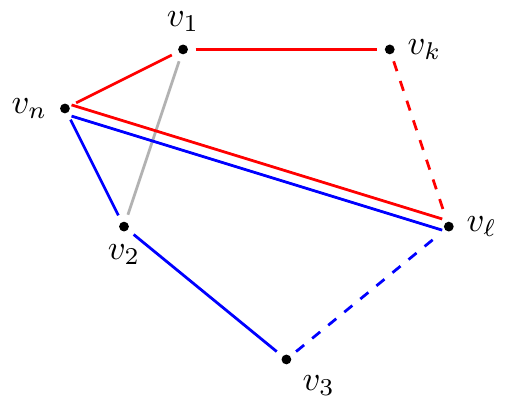}

    \caption{In the figure, $v_1v_2\cdots v_k$ is the chord-free cycle $\gamma$ inside $G_{n-1}$. We perform the edge-split operation with $(v_1, v_2)$ the selected edge and $v_\ell$ the selected node that belongs to the cycle. Then, there are two chord-free cycles after the operation, namely $v_1v_nv_\ell v_{\ell-1}\cdots v_1$ in red and $v_2v_3 \cdots  v_\ell v_nv_2$ in blue.}
    \label{fig:twocyclesedgesplit}
\end{figure}

The next result indicates which operations of a Henneberg construction create chord-free cycles of lengths greater than 3. Per the previous Lemma, these operations cannot be used to construct a TLG.

\begin{Lemma}\label{lem:badoperation}
Let $G'$ be a Laman graph. Let $G$ be the graph obtained by performing on $G'$  {\em one} Henneberg step taken from the following options:
\begin{enumerate}
\item Node-add operation  connecting a new node to two non-adjacent nodes;
\item Edge-split operation  splitting a non-simple edge;
\item Edge-split operation performed on an edge $e = (v_i, v_j)$ and a node $v_k$ such that the three nodes $\{v_i, v_j, v_k\}$ do not form a triangle in $G'$. 
\end{enumerate}
Then, the resulting graph $G$ is not triangulated.  
\end{Lemma}

\begin{proof}
We deal with the three options individually. 

{\it Option 1:} 
Denote by $v_n$ the new node and $v_{i},v_{j}$ the existing nodes in $G'$ that are connected to $v_n$ via the one-step Henneberg construction.  By assumption, $v_i$ and $v_j$ are not adjacent. Let $\gamma=v_{i}\cdots v_{j}$ be the shortest path in $G'$ joining these two nodes. Then, $\omega:=\gamma \vee v_nv_{i}$ is a cycle in $G$. This cycle is chord-free because $\gamma$ is a shortest path. Furthermore, since  $v_{i}$ and $v_{j}$ are not adjacent, the length of $\omega$ is greater than $3$. 

For the remaining two options, we let $(v_{i},v_{j})$ and $v_k$ be the selected edge and node of $G'$, respectively, for the edge-split operation.   

{\it Option 2:} Since $(v_i,v_j)$ is not simple, there exist two distinct triangles $\Delta = \{v_i, v_j, v_\ell\}$ and $\Delta'=\{v_i, v_j, v_m\}$ in $G'$ that share the edge. 
Because $G'$ is a Laman graph, $(v_\ell, v_m)$ cannot be an edge of $G'$. To see this, note that if $(v_\ell, v_m)$ is an edge, then there will be $6$ edges in the subgraph of $G'$ induced by the four nodes $v_i,v_j,v_\ell, v_m$, which violates Laman's condition (Lemma~\ref{lem:lamancondition}). 
But, then, after the edge-split operation, the edge $(v_i, v_j)$ is removed and, hence, 
$v_iv_mv_jv_\ell v_i$ is a chord-free cycle in $G$ of length $4$; see Fig.~\ref{sfig:TLGthopt2}.

\begin{figure}
    \centering
\subfloat[\label{sfig:TLGthopt2}]{
    \includegraphics{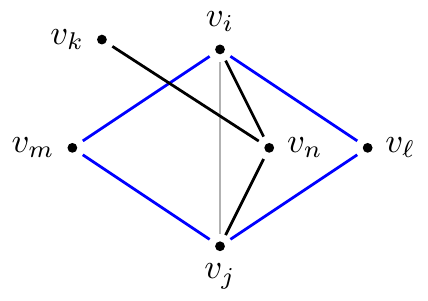}
}\qquad
\subfloat[\label{sfig:TLGthopt3}]{
    \includegraphics{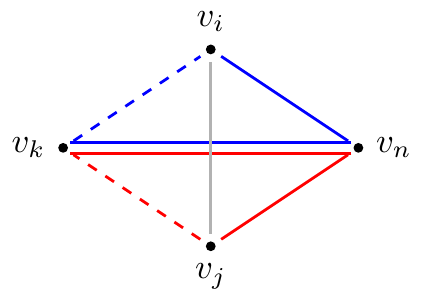}

}
\caption{In Fig.~\ref{sfig:TLGthopt2}, $(v_i,v_j)$ is a non-simple edge shared by  two triangles in $G'$. Performing edge-split operation on $(v_i,v_j)$ with any node $v_k$ in $G'$ ($v_k$ can be $v_m$ or $v_l$) results in a chord-free cycle $v_iv_m v_j v_\ell$ of length $4$.   
In Fig.~\ref{sfig:TLGthopt3}, the dashed lines indicate the shortest path $\gamma$ from $v_i$ to $v_j$ in $G''= G' - (v_i,v_j)$.
Performing edge-split operation on the edge $(v_i,v_j)$ and a node $v_k$ that belong to $\gamma$ yields two chord-free cycles depicted in blue and red, respectively. The total lengths of the two cycles is the length of $\gamma$ plus 4.
}
\label{fig:TLGthopt2and3}
\end{figure}

{\it Option 3:}
Since $G'$ is Laman, it is two-edge-connected. Let $G''$ be obtained from $G'$ by removing the edge $(v_{i},v_{j})$. Then, $G''$ is connected. 
Denote by $\gamma=v_{i}\cdots v_{j}$ the shortest path in $G''$ joining $v_{i}$ to $v_{j}$. The length of $\gamma$ is at least $2$. 
If $v_k$ does not belong to $\gamma$, then $\gamma \vee v_nv_{i}$ is chord-free (because $\gamma$ is a shortest path) and its length is at least $4$.  
We now assume that $v_k$ belongs to $\gamma$. Note that  $\gamma \neq v_{i}v_kv_{j}$ because otherwise these three nodes formed a triangle in $G'$, contradicting our assumption. Hence, the length of $\gamma$ is at least $3$. In this case, the cycle $\omega:=\gamma \vee v_nv_{i}$ has a single chord, namely $(v_k,v_n)$. 
Indeed, on the one hand, $v_n$ is only connected to $v_k$ by construction, so no other chord is incident to $v_n$; on the other hand, the fact that $\gamma$ is a shortest path precludes the existence of a chord between any two nodes in the path. 
This shows that $(v_k,v_n)$ is the only chord in $\omega$, which can thus be split into two cycles of smaller lengths given by $v_{i}\cdots v_k v_n v_{i}$ and $v_{k} \cdots v_{j} v_n v_k$. 
Moreover, the two cycles are chord-free.   
The sum of the lengths of these two cycles is the length of $\gamma$ plus 4, which is at least 7. 
Thus, at least one of the two cycles has its length greater than 3. 
See Fig.~\ref{sfig:TLGthopt3} for illustration.  
\end{proof}

With the above preliminaries, we now prove Theorem~\ref{th:equihennebergrtl}:

\begin{proof}[Proof of Theorem~\ref{th:equihennebergrtl}]
From Lemma~\ref{lem:sufficiency}, every graph obtained by an RHC is a TLG.  
We now show that the converse is also true, i.e., every TLG $G$ can be obtained by an RHC. Let $\cH$ be a Henneberg construction for $G$. Note that $\cH$ can be described by either a sequence of graphs $G_3,\ldots, G_n = G$ along the construction or by a sequence of operations $\cH_3,\ldots, \cH_{q-1}$ applied to these graphs, i.e., operation $\cH_p$ is applied to $G_p$ to obtain $G_{p+1}$. In the sequel, we will use both descriptions. 
To keep the notation simple, we do not make explicit the argument of the operations $\cH_i$. The arguments are an edge for a node-add operation, and an edge and a node for an edge-split operation. 
Note that an operation $\cH_p$ can be applied to any $G_{q}$ as long as $G_q$ contains the selected edge (and node if $\cH_p$ is an edge-split). 

By Lemma~\ref{lem:tlsubgraphs}, the Henneberg construction $\cH$ can only contain operations of the following two types:  (1) node-add operation as described in the RHC, or (2) edge-split operation performed on a simple edge $(v_i,v_j)$ and a node $v_k$ so that $\{v_i,v_j,v_k\}$ is a triangle. 
We now prove that the operations of type (2) can be translated into operations of type (1), thus showing that any Henneberg construction yielding a TLG can be replaced by an RHC.

Let $q \ge 4$ be the smallest integer such that an edge-split operation $\cH_{q-1}$ as in (2) above is used on $G_{q-1}$ to obtain $G_q$. 
Then, the triangle $\{v_i,v_j, v_k\}$ belongs to $G_{q-1}$ and, furthermore, $G_{q-1}$ is obtained by using only the node-add operation.  Hence, each $\cH_p$, for $p = 3,\ldots, q - 2$, is necessarily of type (1) and the truncated sequence  $\cH_3\cdots\cH_{q-2}$ is in fact an RHC. 
The starting triangle of the RHC might not be $\{v_i,v_j,v_k\}$. However, by Prop.~\ref{prop:hennebergfromany}, we can always find another RHC that starts with $\{v_i,v_j, v_k\}$ and yields $G_{q-1}$. We can thus assume, without loss of generality, that $G_3 = \{v_i,v_j,v_k\}$ is the starting triangle.  
It is important to note that because $(v_i,v_j)$ is {\it simple} in $G_{q-1}$, no operation $\cH_p$, for $3\le p \le q - 2$, selects the edge $(v_i,v_j)$, since it already belongs to the triangle $G_3=\{v_i,v_j,v_k\}$.

Next, we will exhibit an RHC $\cH' = \cH'_3\cdots \cH'_{q-1}$ that yields $G_q$.
Starting from $G'_3 = \{v_i,v_k,v_q\}$, the operation $\cH'_3$  simply adds the node $v_j$ and edges $(v_j,v_k)$ and $(v_j,v_q)$, so $G'_4$ is comprised of two triangles, namely $\{v_i,v_k,v_q\}$ and $\{v_j,v_k,v_q\}$. 
Note that $G'_4$ can be also obtained by applying the sequence $\cH_3\cH_{q-1}$ to the triangle $\{v_i,v_j,v_k\}$.  
Now, since $\cH_3$ was applied to the triangle $\{v_i,v_j,v_k\}$, but did not select edge $(v_i,v_j)$,  we can apply the same operation to $G'_4$ to obtain $G'_5$, i.e., we let $\cH'_4 := \cH_3$. Next, observe that the edge selected by $\cH_4$ belongs to $G_4$ and it is not $(v_i,v_j)$. Because $G_4$ (resp. $G'_5$) is obtained from $G_3$ (resp. $G'_4$) using the same operation $\cH_3$, the edge selected by $\cH_4$ belongs to $G'_{5}$. We can thus set $\cH'_5:=\cH_4$. 
Furthermore, note that $G'_5$ can be obtained by applying the sequence $\cH_3\cH_4\cH_{q-1}$ to the triangle $\{v_i,v_j,v_k\}$. 
Applying the above arguments iteratively, we conclude that for any $p = 4,\ldots,q-1$, the graph $G'_p$ obtained by applying $\cH'_3\cH'_4\cdots \cH'_{p}= \cH'_3 \cH_3\cdots \cH_{p-1}$ to the triangle $G'_3 = \{v_i,v_k,v_q\}$ is the same as the graph obtained by applying $\cH_3\cdots \cH_{p-1}\cH_{q-1}$ to the triangle $G_3= \{v_i,v_j,v_k\}$. In particular, for $p = q-1$, we obtain that $G'_q = G_q$. 

\begin{figure}
    \centering
    \subfloat[\label{sfig:rhc1}]{
    \includegraphics{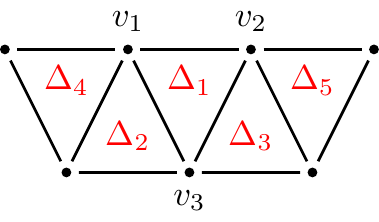}

}
\qquad
\subfloat[\label{sfig:rhc2}]{
    \includegraphics{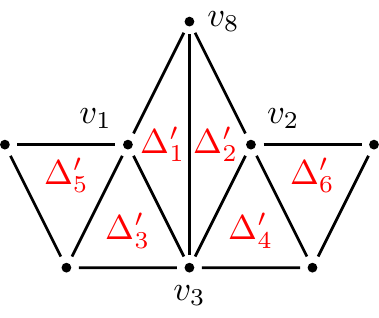}

}
    \caption{The TLG depicted in Fig.~\ref{sfig:rhc1} corresponds to $G_{q-1}$ (here, $q = 8$) in the proof and is obtained via an RHC starting with $\Delta_1$. The indices of the triangles reflect the order of which they have been added into the graph.
    The operation $\cH_4$ selects $(v_1,v_3)$, and $\cH_5$ selects $(v_2,v_3)$, etc. The operation $\cH_7$ applied to $G_7$ is an edge-split operation on the edge $(v_1,v_2)$ and the node $v_3$. It yields the TLG $G_8$ depicted in Fig.~\ref{sfig:rhc2}. Starting with $G'_3=\{v_1,v_3,v_8\}$, we apply the node-add operation $H'_3$ by adding node $v_2$ and edges $(v_2,v_8)$ and $(v_3,v_8)$ to obtain $G'_4$.  
        Then, applying the sequence of operations $\cH_4\cdots \cH_7$, we add triangles $\Delta'_3,\cdots, \Delta_6'$ into the graph and obtain an RHC for $G_8$. 
    }
    \label{fig:my_label}
\end{figure}

Now, replace the original Henneberg construction $\cH_3\cdots\cH_{n-1}$ with the one $\cH'_3\cdots\cH'_{q-1}\cH_q\cdots\cH_{n-1}$. By doing so, we reduce by one the number of edge-split operations. One can repeatedly apply the above arguments until all the edge-split operations in the original Henneberg construction are removed. This process ends with an RHC that yields the graph $G$.     
\end{proof}
\end{document}